\documentclass{preprint}
\usepackage{times}
\usepackage{microtype}
\usepackage{amsmath}
\usepackage{mathrsfs}
\usepackage{mhequ}
\usepackage{mhenvs}
\usepackage{mhsymb}

\definecolor{darkred}{rgb}{0.9,0.1,0.1}

\newop{diag}

\def\DD{\mathcal{D}}
\def\D{\mathscr{D}}
\def\F{\mathscr{F}}
\def\XX{{\boldsymbol X}}
\def\PPsi{{\boldsymbol\Psi}}
\def\E{\mathbf{E}}
\def\/{|\!|\!|}

\begin{document}

\title{Rough Stochastic PDEs}
\author{M.~Hairer}
\institute{Mathematics Department, University of Warwick
 \\ \email{M.Hairer@Warwick.ac.uk}}
\titleindent=0.65cm

\maketitle
\thispagestyle{empty}

\begin{abstract}
In this article, we show how the theory of rough paths can be used to provide a notion
of solution to a class of nonlinear stochastic PDEs of Burgers type that exhibit too high spatial roughness for 
classical analytical methods to apply. In fact, the class of SPDEs that we consider is genuinely
ill-posed in the sense that different approximations to the nonlinearity may converge to different limits. 
Using rough paths theory, a pathwise notion of solution to these SPDEs is formulated, and we show that 
this yields a well-posed problem, which is stable under a large class of perturbations, including
the approximation of the rough driving noise by a mollified version and the addition of hyperviscosity. 

We also show that under certain structural assumptions on the coefficients, the SPDEs under consideration
generate a reversible Markov semigroup with respect to a diffusion measure that can be given
explicitly.
\end{abstract}

\section{Introduction}

This article is devoted to the study of the following class of Burgers-like SPDEs:
\begin{equ}[e:Burgers]
du = \d_x^2 u\,dt + f(u)\,dt + g(u)\,\d_x u\,dt + \sigma\,dW(t)\;.
\end{equ}
Here, the spatial variable $x$ takes values in $[0,2\pi]$, the linear operator $\d_x^2$ is endowed with periodic boundary
conditions, $u$ takes values in $\R^n$, and $f\colon \R^n \to \R^{n}$, $g\colon \R^n \to \R^{n\times n}$ are $\CC^\infty$ 
functions. We assume that the driving noise $W$ gives rise to space-time white noise; in other words that $W$ is a standard 
cylindrical Wiener process on $\L^2([0,2\pi], \R^n)$ \cite{DaPrato-Zabczyk92}. One motivation for studying such equations arises from the theory
of path sampling: for $f$ and $g$ of some specific form, \eref{e:Burgers} does \textit{formally} arise as a gradient
system with the law of a diffusion process as invariant measure, see \cite{HairerStuartVoss07} and Section~\ref{sec:IM} below.

The problem with \eref{e:Burgers} that we address in this article is that of making sense of the nonlinearity $g(u)\, \d_x u$ in equations of this type.
To appreciate the difficulty of the problem, we note that the solution $\psi$ to the linearised equation
\begin{equ}[e:linear]
d\psi = \d_x^2 \psi\,dt  + \sigma\,dW(t)\;,
\end{equ}
is not differentiable in $x$ for fixed $t$. Actually, these solutions have a spatial regularity akin to the temporal
regularity of Brownian motion: they are almost surely $\alpha$-H\"older continuous for every $\alpha < {1\over 2}$, but
not more in the sense that they are almost surely \textit{not} ${1\over 2}$-H\"older continuous \cite{Walsh}.

This usually doesn't cause any serious problem: the standard procedure in this case is to consider weak (in the PDE sense)
solutions of the form
\begin{equ}
d\scal{\phi,u} = \scal{\d_x^2 \phi,u}\,dt + \scal{\phi,f(u)}\,dt  + \scal{\phi,g(u)\,\d_x u}\,dt + \sigma\,\scal{\phi,dW(t)}\;,
\end{equ}
for sufficiently regular test functions $\phi$ and to make sense of the term $\scal{\phi,g(u)\,\d_x u}$ by performing
one integration by parts. However, this is only possible if there exists a function $G\colon \R^n \to \R^n$ such that 
$g = DG$. Assuming the existence of such a function $G$ 
would impose non-trivial structural conditions on $g$ as soon as $n > 1$, which is not something that we wish to do.
Now if it were the case that, for fixed $t>0$, $u$ was $\alpha$-H\"older continuous for some exponent $\alpha$ strictly
\textit{greater} than ${1\over 2}$, then we could rewrite the nonlinearity in the suggestive form
\begin{equ}[e:nonlin]
\scal{\phi,g(u)\,\d_x u} = \int_0^{2\pi} \phi(x) g\bigl(u(x)\bigr)\,du(x)
\end{equ} 
and interpret this integral as a simple Riemann-Stieltjes integral. By Young's theory of integration \cite{Young}, this
expression would indeed be well-defined in this case. Unfortunately, as already mentioned, we expect our solutions to fall
just slightly short of this kind of regularity, so that there is a priori no obvious way in which to make sense of \eref{e:nonlin}.
From this perspective, the problem at hand is very strongly reminiscent of the problem of making sense of solutions to ordinary 
stochastic differential equations. Actually, similarly to the case of SDEs, different numerical approximations to \eref{e:Burgers}
converge to different solutions, which differ by a correction term similar to the classical It\^o-Stratonovich correction term,
see \cite{Numeric} for a numerical exploration of this phenomenon.

Motivated by this observation, let us try to apply the standard theory of stochastic integration to this problem.
For this, we need to first specify what type of stochastic integral we wish to consider. Since we would like to recover
the usual concept of weak / mild solutions for the Burgers equations in the case where $g$ is a total derivative,
it is natural to look for a kind of `Stratonovich integral' interpretation of \eref{e:nonlin}.
Since we expect $u$ to behave like $\psi$ at `small scales', it is arguably sufficient to make sense of the expression
\begin{equ}[e:nonlin2]
\scal{\phi,g(u)\,\d_x \psi} = \int_0^{2\pi} \phi(x) g\bigl(u(x)\bigr)\circ d\psi(x)\;.
\end{equ} 
This seems promising since it can easily be seen that for fixed $t$, the law of $\psi$ differs from that of a Brownian bridge
only by the addition of some random $\CC^\infty$ function. The problems with this approach seem twofold:
\begin{enumerate}
\item There is no `arrow of time'. In particular, the process $g\bigl(u(x)\bigr)$ is not adapted with respect to the filtration
generated by $\psi$.
\item For any fixed $t>0$, both $\psi(\cdot,t)$ and $u(\cdot,t)$ have a very complicated dependence on the driving space-time
white noise for times $s < t$. This would make it a highly non-trivial task to develop a Malliavin calculus of $u$ with respect to $\psi$
and to study the dependence of this calculus on the time parameter $t$.
\end{enumerate}
All of these problems can be solved in an elegant way if the integral in \eref{e:nonlin2} can be interpreted in a pathwise sense.
The theory of rough paths developed by Lyons \cite{Lyons} provides just such an interpretation! The twist here is that 
we will use the theory of rough paths in order to make sense of a driving noise that leads to solutions that are
rough \textit{in space} rather than in time.

One may wonder at this stage whether the notion of solution to \eref{e:Burgers} given by rough paths theory is in any way
natural. This question will be answered by the affirmative in two different ways. First, it is natural to consider a smoothened
version of \eref{e:Burgers} where the noise $W$ is hit by a mollifier with lengthscale $\eps>0$ and to study the limit
$\eps > 0$. We will see in Section~\ref{sec:stability} that the stability properties of our solution, together with known approximation
 results for Gaussian rough paths, imply that the sequence of classical solutions obtained in this way does indeed
 converge as $\eps \to 0$ to the solution constructed in this article.
Secondly, we will come back to the original motivation for the study of \eref{e:Burgers}, which is 
to provide an SPDE with invariant measure given by a certain diffusion process. We will show in
Section~\ref{sec:IM} that it is indeed the case that if we consider \eref{e:Burgers} with a particular structure for
the nonlinearities $f$ and $g$ derived formally in \cite{HairerStuartVoss07}, then the process constructed in this article is reversible
with respect to the expected invariant measure.

It is of course not the first time that the theory of rough paths has been applied to stochastic PDEs. To our knowledge,
three groups of authors have considered such problems in quite different contexts. Friz and coauthors showed in
\cite{PeterMich,Oberhaus} that rough paths theory can be used to provide meaning (and solutions) to a class of nonlinear stochastic
PDEs via the method of stochastic characteristics. This is essentially a variant of the type of problems that have been
considered by Souganidis and Lions \cite{LionsSoug}, and the emphasis in these problems is the treatment of
temporally rough driving signals. Concurrently, Gubinelli and Tindel developed a theory of stochastic PDEs
driven by rough paths which allows to treat semilinear problems of `Da Prato \& Zabczyk type', see \cite{MaxSam}. There, the 
emphasis is not just on treating temporally rough driving noise, but also on understanding the interplay between
temporal and spatial regularity. This theory is based on the ideas developed in \cite{Max}, combined with the insights 
obtained in the more regular case in \cite{GubLejTin}, but it relies on classical Sobolev calculus to treat the spatial
roughness of the solutions. Finally, a more recent result was obtained by Teichmann \cite{Josef}, 
where Sz\H okefalvi-Nagy's dilation theorem for contraction semigroups is used to provide a simple and elegant way of constructing
solutions to a class of semilinear SPDEs when the corresponding linear problem generates a semigroup of contractions
on a Hilbert space. We also refer to the works \cite{GubKdV,BrGuNe} for examples of \textit{deterministic} PDEs that can
be tackled using rough paths theory.

The main novelty of the present work is the ability to give meaning to a class of stochastic PDEs 
such that the \textit{deterministic part} of the equation does not have any classical meaning.
While this has been achieved in a number of equations using renormalisation techniques 
\cite{MR815192,MR1462228,MR1743612,MR2016604}, to the 
best of our knowledge, this is the first time that rough paths theory is used in such an endeavour. 
The advantage of rough paths theory in that context is that it allows to treat nonlinearities that do not exhibit a
`polynomial' structure, as is required by renormalisation techniques.
It is also the first time that rough path theory is used to provide meaning
to an equation which is classically ill-posed due to a lack of \textit{spatial} regularity, rather than a lack of
temporal regularity.

The remainder of this article is organised as follows. First, in Section~\ref{sec:roughpaths}, we give a short overview of those
elements of rough path theory that are being used in this work. While this is of course by no means a general introduction
to the theory (we refer for this to the monographs \cite{MR2036784,PeterBook} and the lecture notes \cite{MR2314753}), 
it is intended to be sufficiently self-contained so that even a reader without prior knowledge of rough paths theory
should be able to follow the subsequent arguments. Section~\ref{sec:defSol} provides the definition of a solution to
\eref{e:Burgers}, as well as the proof that this equation is locally well-posed (globally if $f$ and $g$ are sufficiently bounded)
and that its solutions are stable with respect to perturbations in the initial condition and the driving noise.
In Section~\ref{sec:IM}, we then show that under the structural assumptions derived in \cite{HairerStuartVoss07}, one
can explicitly exhibit an invariant measure for \eref{e:Burgers}, and the corresponding Markov process is reversible.
Finally, Section~\ref{sec:uniformExp} contains a uniform exponential integrability result which is essential 
in the proofs of Section~\ref{sec:IM}.

\subsection{Notations}

We denote by $\CC^\alpha$ the space of all $\alpha$-H\"older continuous functions on $[0,2\pi]$ and by $\|\cdot\|_\alpha$
the corresponding seminorm, namely
\begin{equ}
\|u\|_\alpha = \sup_{x\neq y} {|u(x) - u(y)|\over |x-y|^\alpha}\;.
\end{equ}
We will also make a slight abuse of notation by writing $\|\cdot\|_\infty$ for the supremum norm and 
we set $\|u\|_{\CC^\alpha} = \|u\|_\alpha + \|u\|_\infty$ which, on bounded intervals, is also equivalent to $\|u\|_\alpha + |u(0)|$.
For integer values of $n$, we set $\|u\|_{\CC^n} = \sum_{k=0}^n \|D^k u\|_\infty$, where $D^ku$ denotes the $k$th 
derivative of $u$.

\subsection*{Acknowledgements}

{\small
I am grateful to Peter Friz, Massimilliano Gubinelli, Terry Lyons, and Josef Teichmann for
introducing the theory of rough paths to me and for several fruitful discussions about this work.
Special thanks are due to Andrew Stuart for introducing me to the path sampling
problems that were the original motivation for looking at the equations considered in this article.
Finally, I would like to thank Hendrik Weber for his careful reading of a draft version of the manuscript.
Financial support was kindly provided by the EPSRC through grants EP/E002269/1 and EP/D071593/1, as
well as by the Royal Society through a Wolfson Research Merit Award.
}

\section{Elements of rough path theory}
\label{sec:roughpaths}

We will mostly make use of the notations introduced by Gubinelli in \cite{Max} since the estimates given in that work 
seem to be the ones that are most 
suitable for the present undertaking. This is because Gubinelli essentially builds a theory of integration
for quite general integrands against a given rough path, whereas Lyons mostly considers integrands that are
the composition of a smooth (local!) function with the rough path. This restriction could in principle be overcome by a 
slight reformulation of the problem (just as it can be overcome when one wishes to use the theory to solve SDEs), 
but this appears to be more cumbersome in our setting.

We denote by $\CC_2([0,T], \R^d)$ the space of continuous functions
from $[0,T]^2$ into $\R^d$ that vanish on the diagonal. Very often, we will omit the time 
interval $[0,T]$ and the target space $\R^d$ in our notations  for the
sake of simplicity. We also define a difference operator $\delta\colon \CC \to \CC_2$ by
\begin{equ}
\delta X_{s,t} = X_t - X_s\;.
\end{equ}

A rough path on an interval $[0,T]$ then consists of two parts: a continuous function
$X \colon [0,T] \to \R^n$, as well as a continuous `area process' $\XX \colon [0,T]^2 \to \R^{n \times n}$ such that
$\XX_{t,t} = 0$ for every $t$ and such that the algebraic relations
\begin{equ}[e:constr]
\XX^{ij}_{s,t} - \XX^{ij}_{u,t} -\XX^{ij}_{s,u} =  \delta X^i_{s,u}\delta X^j_{u,t}\;,
\end{equ}
hold for every triple of times $(u,s,t)$ and every pair of indices $(i,j)$.
One should think of $\XX$ as `postulating' the value of the quantity
\begin{equ}[e:intX]
\int_s^t \delta X^i_{s,r}\,dX^j_r \eqdef \XX^{ij}_{s,t} \;,
\end{equ}
where we take the right hand side as a \textit{definition} for the left hand side. (And not the other way around!)
Note that the algebraic relations \eref{e:constr} are by themselves not sufficient to determine $\XX$ as a function
of $X$. Indeed, for any 
matrix-valued function $F$, the substitution $\XX^{ij}_{s,t} \mapsto \XX^{ij}_{s,t} + F^{ij}_t - F^{ij}_s$
leaves the left hand side of \eref{e:constr} invariant.
The aim of imposing \eref{e:constr} is to ensure that \eref{e:intX} does indeed behave
like an integral when considering it over two adjacent intervals.

For $\alpha \in (0,{1\over 2})$, we will denote by $\DD^\alpha$ the space of those rough paths $(X,\XX)$
such that $X \in \CC^\alpha$ and
\begin{equ}[e:Holder]
\|\XX\|_{2\alpha} :=  \sup_{s \neq t \in [0,T]} {|\XX_{s,t}|\over |t-s|^{2\alpha}} < \infty\;.
\end{equ}
At this stage, it is important to note that while it is a closed subset of a vector space, the space 
$\DD^\alpha$ is \textit{not} itself a vector space because of the nonlinear
constraint \eref{e:constr}. One rather unpleasant consequence of this fact is that the natural norm on $\DD^\alpha$ 
given by $\|X\|_{\CC^\alpha} + \|\XX\|_{2\alpha}$ does not
reflect its geometry, since the natural
dilatation on $\CD^\alpha$ is given by $(X,\XX) \mapsto (\lambda X, \lambda^2 \XX)$. Note also that the quantities defined in
\eref{e:Holder} are merely seminorms since they vanish for constants.


\subsection{Controlled rough paths}

Another important notion taken from \cite{Max} is that of a path $Y$ \textit{controlled} by a rough path $X$.
Given a rough path $X \in \DD^\alpha([0,T], \R^d)$, we say that a pair of functions $(Y,Y') \in \CC^\alpha_X([0,T], \R^m)$ 
is a rough path controlled by $X$ if 
$Y \in \CC^\alpha([0,T], \R^m)$,
$Y' \in \CC^\alpha([0,T], \R^{m\times d})$, and the `remainder term' $R$ given by
\begin{equ}[e:defYR]
R_{s,t} = \delta Y_{s,t} - Y'_s\, \delta X_{s,t}\;,
\end{equ}
satisfies $\|R\|_{2\alpha} < \infty$. Here, $R_{s,t} \in \R^m$ and the second term is  a matrix-vector multiplication. 
We endow the space $\CC^\alpha_X$ with the norm
\begin{equ}[e:normcontrol]
\|(Y,Y')\|_{X,\alpha} = \|Y\|_{\CC^\alpha} + \|Y'\|_{\CC^\alpha} + \|R\|_{2\alpha}\;.
\end{equ}
Note that since we assumed that $X$ is $\alpha$-H\"older continuous, it immediately follows from these definitions
that the same is true for $Y$ with
\begin{equ}
\|Y\|_\alpha \le C\|R\|_{2\alpha} + \|Y'\|_\infty \|X\|_\alpha\;.
\end{equ}
The term $\|Y\|_{\CC^\alpha}$ in \eref{e:normcontrol} is therefore used only to control the supremum of $Y$.

\begin{remark}
We will sometimes make an abuse of notation and simply write $\|Y\|_{X,\alpha}$ instead
of the more correct expression $\|(Y,Y')\|_{X,\alpha}$. Since $Y'$ will always be constructed from $X$ and
regular functions by using the rules laid out in the next subsections, this will hopefully not cause any confusion.
\end{remark}

Note that in general, there could be many `derivative processes' $Y'$ associated to a given path $Y$. 
However,  if for some given $s \in (0,T)$
there exists a sequence of times $t_n \to s$ such that $|X_{t_n} - X_s| / |t_n-s|^{2\alpha} \to \infty$, then $Y'_s$ is uniquely
determined from $Y$ by \eref{e:defYR} and the condition that $\|R\|_{2\alpha} < \infty$. 
In most cases of interest, such as when it is given by the sample path of a (fractional) Brownian motion,
the function $X$ will have this property at a dense set of points, thus determining $Y'$ uniquely as a function of $Y$.

In the sequel, we will sometimes omit to explicitly mention the derivative process $Y'$. We hope that this will not cause any ambiguity 
since all the controlled paths that we are going to consider will be constructed using the following list of operations.

\subsubsection{Canonical lift of $X$}
\label{sec:canlift}

It is easy to see that the process $X$ itself can be interpreted as a process `controlled by $X$'. 
Indeed, we can identify $X$ with the element $(X, I) \in \CC_X^\alpha$, where $I$ is the process which is 
equal to the identity matrix for all times.

\subsubsection{Lifting of regular functions.}
\label{sec:lift}

There is a canonical embedding $\iota\colon \CC^{2\alpha} \hookrightarrow \CC^\alpha_X$ given by $\iota Y = (Y,0)$, 
since in this case $R = \delta Y$ does indeed satisfy $\|R\|_{2\alpha} < \infty$ (recall that we are only interested in the case
$\alpha < {1\over 2}$). 
If one actually has $Y \in \CC^1$, then one can define the integral of $Y$ against $X$ by setting 
\begin{equ}[e:defintYX]
Z_t = \int_0^t Y_s\otimes dX_s = \dot Y_t \otimes X_t - \dot Y_0 \otimes X_0 -  \int_0^t \dot Y_s \otimes X_s\,ds\;,
\end{equ}
where $\dot Y$ denotes the time derivative of $Y$. One can check quite easily that this integral has the property that
$\bigl(Z, Y\otimes I\bigr)$, where $I$ is the identity matrix, is itself a controlled rough path belonging to $\CC^\alpha_X$.

\subsubsection{Composition with regular functions.}
\label{sec:compos}

Let $\phi \colon \R^m \times [0,T] \to \R^n$ be a function which is uniformly $\CC^2$ in its first argument (i.e.\ $\phi$
is bounded and both $D_y \phi$ and $D_y^2 \phi$ are bounded, where $D_y$ denotes the derivative with respect to
the first argument) and uniformly 
$\CC^{2\alpha}$ in its second argument. Let furthermore $(Y,Y')\in \CC^\alpha_X([0,T], \R^m)$, then one can define
a controlled path $(\phi(Y), \phi(Y)')\in \CC^\alpha_X([0,T], \R^n)$ by
\begin{equ}[e:defphiY]
\phi(Y)_t = \phi(Y_t, t)\;,\qquad
\phi(Y)_t' = D_y \phi(Y_t, t)Y'_t\;.
\end{equ}
(Here, the path $\phi'(Y)Y'$ is to be interpreted as the pointwise an $n\times m$ and an $m \times d$ 
matrix-valued path.)
It is straightforward to check that the corresponding remainder term does indeed satisfy the required bound.
It is also straightforward to check that this definition is consistent in the sense that $(\phi \circ \psi)(Y,Y') = \phi(\psi(Y,Y'))$.
Furthermore, we have the bound:

\begin{lemma}\label{lem:comp}
Let $\phi$ be as above, let $(Y,Y')\in \CC^\alpha_X([0,T], \R^m)$, and let $(\phi(Y), \phi(Y)')\in \CC^\alpha_X([0,T], \R^n)$ 
be given by \eref{e:defphiY}. Then, there exists a constant $C$ such that one has the bound
\begin{equ}
\|\phi(Y)\|_{X,\alpha} \le C \bigl(\|D_y^2\phi\|_\infty + \|\phi\|_\infty + \|\phi\|_{2\alpha;t}\bigr) \bigl(1+\|Y\|_{X,\alpha}\bigr)^2\;,
\end{equ}
where we denote by $\|\phi\|_{2\alpha;t}$ the supremum over $y$ of the $2\alpha$-H\"older norm of $\phi(y,\cdot)$. 
\end{lemma}

\begin{proof}
We start by showing that there exists a constant $C$ such that 
\begin{equ}[e:interp]
\|D_y\phi\|_{\alpha;t}^2 \le C \|D^2\phi\|_\infty \|\phi\|_{2\alpha;t}\;.
\end{equ}
We consider the case $m=n=1$, the general case follows in a similar way. There are two times $s \neq t$ and
a point $x$ such that
\begin{equ}
\phi'(x,s) - \phi'(x,t) = \|\phi'\|_{\alpha;t} |t-s|^\alpha \eqdef \eps\;.
\end{equ}
Therefore, for $y$ such that $|y-x| \le \eps/(4\|\phi''\|_\infty)$, we have
\begin{equ}
\phi'(y,s) - \phi'(y,t) \ge {\eps\over 2}\;.
\end{equ}
Integrating this inequality from $x$ to $y$, we obtain
\begin{equ}
\phi(y,s) - \phi(x,s) - \phi(y,t) + \phi(x,t) \ge  {\eps^2 \over 8\|\phi''\|_\infty}\;,
\end{equ}
so that 
\begin{equ}
{\|\phi'\|_{\alpha;t}^2 |t-s|^{2\alpha} \over 8\|\phi''\|_\infty} = {\eps^2 \over 8\|\phi''\|_\infty}\le 2\|\phi\|_{2\alpha;t} |t-s|^{2\alpha}\;,
\end{equ}
which is precisely the claim \eref{e:interp}.

It follows from \eref{e:defphiY} and elementary properties of the H\"older norms that
\begin{equs}
\|\phi(Y)\|_\alpha &\le C (\|\phi\|_\infty + \|D_y\phi\|_\infty) \|Y\|_{\CC^\alpha} + \|\phi\|_{\alpha;t}\|Y\|_\infty\;,\\
\|\phi(Y)'\|_\alpha &\le  C (\|D\phi\|_\infty + \|D_y^2\phi\|_\infty) \|Y\|_{\CC^\alpha} \|Y'\|_{\CC^\alpha}  + \|D_y\phi\|_{\alpha;t} \|Y'\|_\infty\\
 &\le  C (\|\phi\|_\infty + \|D_y^2\phi\|_\infty) \|Y\|_{X,\alpha}^2  + C \sqrt{\|D^2\phi\|_\infty \|\phi\|_{2\alpha;t}} \|Y'\|_\infty\;.
\end{equs}
Concerning the remainder, we have the bound
\begin{equ}
\bigl|\phi(Y_t, t) - \phi(Y_s, s) - D_y \phi(Y_s, s)\delta Y_{s,t} \bigr| \le \|\phi\|_{2\alpha;t} |t-s|^{2\alpha} + {1\over 2}\|D_y^2 \phi\|_\infty |\delta Y_{s,t}|^2\;.
\end{equ}
Since on the other hand $|\delta Y_{s,t} - Y'_s \delta X_{s,t}| = R_{s,t}$ by definition, we then have
for the remainder term  $R^\phi$ of the controlled rough path $(\phi(Y), \phi(Y)')$ the bound
\begin{equ}
|R^\phi_{s,t}| \le \|D_y \phi\|_\infty |R_{s,t}| + \|\phi\|_{2\alpha;t} |t-s|^{2\alpha} + {1\over 2}\|D_y^2 \phi\|_\infty \|Y\|_\alpha^2 |t-s|^{2\alpha}\;.
\end{equ}
The claim now follows from the assumptions on $\phi$.
\end{proof}

In particular, this shows that if $f \in \CC^{2\alpha}$ and $(Y,Y')\in \CC^\alpha_X$, then
both $f\cdot Y$ and $f + Y$ are well-defined elements of $\CC^\alpha_X$. In that case, 
one can slightly improve over the general bound given in Lemma~\ref{lem:comp}, namely one has
\begin{equs}
\|(fY, fY')\|_{X,\alpha} &\le 2 \|f\|_{\CC^{2\alpha}}\|(Y,Y')\|_{X,\alpha}\;,\label{e:prodfY}\\
\|(f+Y, Y')\|_{X,\alpha} &\le 2\|f\|_{\CC^{2\alpha}} + \|(Y,Y')\|_{X,\alpha}\;.
\end{equs}
It also shows immediately that
$\CC^\alpha_X$ is an algebra for every reference rough path $X$.

\subsection{Integration of controlled rough paths.}

The aim of this section is to give a meaning to the expression $\int Y_t \otimes dX_t$, when 
$X \in \CD^\alpha$ and $(Y,Y') \in \CC_X^\alpha$. A natural approach would be to try to define
it as a limit of Riemann sums, that is 
\begin{equ}[e:defint]
\int_0^1 Y_t \otimes dX_t = \lim_{|\CP| \to 0} \sum_{[s,t] \in \CP} Y_s \otimes \delta X_{s,t}\;,
\end{equ}
where $\CP$ denotes a partition of $[0,1]$ (interpreted as a finite collection of intervals) and $|\CP|$ denotes
the length of the largest element of $\CP$. Unfortunately, this does not converge in general. The next best approximation
to the integral is given by making use of the approximation $Y_t \approx Y_s + Y'_s \,\delta X_{s,t}$ suggested  by
\eref{e:defYR} and combining this with \eref{e:intX}. This suggests that instead of \eref{e:defint}, one should rather define the integral as
\begin{equ}[e:defintfinal]
\int_0^1 Y_t \otimes dX_t = \lim_{|\CP| \to 0} \sum_{[s,t] \in \CP} \bigl(Y_s\otimes  \delta X_{s,t} + Y'_s \,\XX_{s,t}\bigr)\;.
\end{equ}

With these notations at hand, we quote the following result, which is a slight reformulation of \cite[Prop~1]{Max}:

\begin{theorem}\label{theo:integral}
Let $(X,\XX) \in \DD^\alpha$ for some $\alpha > {1\over 3}$ and fix $T>0$. Then, the map 
\begin{equ}[e:defInt]
(Y, Y') \mapsto \Bigl(\int_0^\cdot Y_t\otimes dX_t, Y\otimes I\Bigr)\;,
\end{equ}
with the integral defined as in \eref{e:defintfinal} is continuous 
from $\CC^\alpha_X$ to $\CC^\alpha_X$ and one has the bound
\begin{equ}[e:boundInt]
\Bigl\|\int_0^\cdot \delta Y_{0,t}\otimes dX_t\Bigr\|_\alpha \le
C \bigl(\|X\|_\alpha \|R\|_{2\alpha} + \|\XX\|_{2\alpha}\|Y'\|_{\CC^\alpha}\bigr) \;,
\end{equ}
for some constant $C>0$ depending only on the final time.
If furthermore
$(\bar X, \bar \XX) \in \DD^\alpha$  and $(\bar Y, \bar Y') \in \CC^\alpha_{\bar X}$, then there exists a constant $C$ such that
the bound
\begin{equs}
\Bigl\|\int_0^\cdot \delta Y_{0,t}\otimes dX_t &-\int_0^\cdot \delta \bar Y_{0,t}\otimes d\bar X_t\Bigr\|_\alpha \le
C \|X - \bar X\|_\alpha \bigl(\|R\|_{2\alpha}+ \|\bar R\|_{2\alpha}\bigr) \\
&\qquad + C \|\XX - \bar \XX\|_{2\alpha} \bigl(\|Y'\|_{\CC^\alpha}+ \|\bar Y'\|_{\CC^\alpha}\bigr) \\
&\qquad + C \|R - \bar R\|_{2\alpha} \bigl(\|X\|_\alpha + \|\bar X\|_\alpha\bigr) \\
&\qquad + C \|Y' - \bar Y'\|_{\CC^\alpha}\bigl(\|\XX\|_{2\alpha} + \|\bar \XX\|_{2\alpha}\bigr) \label{e:boundIntDiff}
\end{equs}
holds.
\end{theorem}

\begin{remark}
The bound \eref{e:boundInt} does behave in a very natural way under dilatations.
Indeed, the integral is invariant under the transformation 
\begin{equ}[e:dilat]
(Y,X,\XX) \mapsto (\lambda^{-1}Y, \lambda X, \lambda^2 \XX)\;.
\end{equ}
The same is true for right hand side of \eref{e:boundInt}, since under this dilatation, we also have $(Y', R) \mapsto (\lambda^{-2} Y', \lambda^{-1}R)$. 
\end{remark}

\subsection{Integration against a scaled function}

While the bound \eref{e:boundInt} is well-behaved under \eref{e:dilat}, it is very badly behaved if the integrand is multiplied
by a smooth function that is rescaled in its argument. However, when acting onto the nonlinearity of our equation
with the heat semigroup, this is precisely the type of expression that we encounter, and sharp bounds are essential in order to
obtain the well-posedness of our problem. In this subsection, we give such a bound.
Let $(X,\XX)$ be a rough path belonging to $\DD^\alpha$,
let $(Y,Y') \in \CC_X^\alpha$ be a rough path controlled by $X$, and let $f\colon \R \to \R$ be a $\CC^1$ function such that
\begin{equ}[e:normf]
\|f\|_{1,1} := \sum_{n \in \Z} \sup_{0 \le t \le 1} \bigl(|f(n+t)| + |f'(n+t)|\bigr) < \infty\;.
\end{equ}
We then have the following bound, which is crucial in what follows.

\begin{proposition}\label{prop:convolution}
Let $\alpha > {1\over 3}$.
With the same notations as above, there exists a universal constant $C_\alpha$ such that the bound
\begin{equ}[e:boundint]
\Bigl|\int_0^1 f(\lambda t) Y(t)\,dX(t)\Bigr| \le C_\alpha \lambda^{-\alpha} \|Y\|_{X,\alpha} \|f\|_{1,1} \bigl(\|X\|_\alpha + \|\XX\|_{2\alpha}\bigr) \;,
\end{equ}
holds for all $\lambda \ge 1$.
\end{proposition}

\begin{remark}
Recall that if $Y \in \CC^\alpha_X$, then $t \mapsto f(\lambda t) Y_t$ also belongs to the same space.
Therefore, the integral appearing in \eref{e:boundint} is well-defined in the sense of \cite{Max}.
\end{remark}

\begin{proof}
We assume without loss of generality that $\lambda$ is an integer and we write
\begin{equ}
\int_0^1 f(\lambda t) Y(t)\,dX(t) = \sum_{k=0}^{\lambda -1} \int_0^1 f(t+k) Y_{\lambda,k}(t)\,dX_{\lambda,k}(t)\;,
\end{equ}
where we have set
\begin{equ}
X_{\lambda,k}(t) = X((t+k)/\lambda)\;,\qquad \XX_{\lambda,k}(s,t) = \XX((s+k)/\lambda, (t+k)/\lambda)\;.
\end{equ}
Similarly, the path $Y_{\lambda,k}$ is considered to be controlled by $X_{\lambda,k}$ with $Y_{\lambda,k}'(t) = Y'((t+k)/\lambda)$.
With these notations, it is straightforward to check from the definitions that one has the bounds 
$\|X_{\lambda,k}\|_\alpha \le \lambda^{-\alpha}\|X\|_\alpha$ and $\|\XX_{\lambda,k}\|_{2\alpha} \le \lambda^{-2\alpha}\|\XX\|_{2\alpha}$,
and that furthermore
\begin{equ}
\|Y_{\lambda,k}\|_{X_{\lambda,k},\alpha} \le \|Y\|_{X,\alpha}\;.
\end{equ}
Finally, setting $f_k(t) = f(t+k)$ with $t \in [0,1]$, we can view $f_k$ as a path $F_k \in \CC^\alpha_X$ just as in  Section~\ref{sec:lift}.
Setting 
\begin{equ}
\alpha_k = \|F_k\|_{\CC^{2\alpha}}\;,
\end{equ}
it follows from \eref{e:normf} that there exists a constant $C$ such that $\sum_{k \ge 0} \alpha_k \le C \|f\|_{1,1} < \infty$. 
Combining these bounds, it follows
from Theorem~\ref{theo:integral} and \eref{e:prodfY} that 
\begin{equs}
\Bigl|\int_0^1 f(t+k) Y_{\lambda,k}(t)\,dX_{\lambda,k}(t)\Bigr| &\le C \alpha_k \|Y_{\lambda,k}\|_{X_{\lambda,k},\alpha} \bigl(\|X_{\lambda,k}\|_\alpha + \|\XX_{\lambda,k}\|_{2\alpha}\bigr) \\
&\le C \lambda^{-\alpha} \alpha_k \|Y\|_{X,\alpha}\bigl(\|X\|_\alpha + \|\XX\|_{2\alpha}\bigr) \;,
\end{equs}
so that the requested bound follows at once by summing over $k$.
\end{proof}

\section{Definition of solutions and well-posedness}
\label{sec:defSol}

In this section, we show that it is possible to give a meaning to \eref{e:Burgers} by using rough path theory.
For this, we first denote by $\psi$ the \textit{stationary} solution to the linearised SPDE
\begin{equ}[e:eqlin]
d\psi = (\d_x^2-1) \psi + \sigma dW\;. 
\end{equ}
We then have the following result, which is a straightforward consequence of the results in \cite{CoutQian,GaussI}:

\begin{lemma}\label{lem:defPsi}
For every $\alpha < {1\over 2}$, the stochastic process 
$\psi$ given by \eref{e:eqlin} can be lifted canonically to a process $\Psi \colon \R_+ \to \DD^\alpha$
which has almost surely H\"older continuous sample paths.
\end{lemma}

\begin{proof}
The process $\psi$ is a centred Gaussian process with covariance function given, for fixed $t$, by
\begin{equ}
\E \psi(x,t) \psi(y,t) = K(x-y)\;,
\end{equ}
where $K$ is given by
\begin{equ}
K(x) = {\sigma^2\over 2\pi} \sum_{k \in \Z} { \cos k x \over 1+k^2} = {\sigma^2 \over 2} {\cosh (|x|-\pi) \over \sinh \pi}\;,
\end{equ}
for $x\in [-\pi,\pi]$. It is then extended periodically for the remaining values of $x$. 
In particular, $K$ is $\CC^\infty$ away from the origin, with a jump discontinuity in its first derivative at the origin.

By \cite[Theorem~35]{GaussI}, we conclude that for every fixed $t>0$, the Gaussian process $\psi(\cdot, t)$ can be lifted canonically
to a rough path $\Psi_t$. (See Remark~\ref{rem:canonicalLift} below on the meaning of `canonical' in this context.) 
We stress once again that the spatial variable $x$ plays the r\^ole of `time' here, while $t$ remains fixed!

Furthermore, there exists a constant $C>0$ such that, given any two times $s,t > 0$, we have
\begin{equ}
\E |\psi(x,t) - \psi(x,s)|^2 \le C|t-s|^{1/2}\;.
\end{equ}
This can be seen from the bound
\begin{equ}
\E |\psi(x,t) - \psi(x,s)|^2 = {\sigma^2\over 2\pi} \sum_{k \in \Z} {1- e^{-(1+k^2) |t-s|}\over 1+k^2} \le C \sum_{k \ge 1} |t-s| \wedge k^{-2}\;,
\end{equ}
and bounding this sum by an integral.
By \cite[Theorem~37]{GaussI}, this shows that there exists an exponent $\theta > 0$ and constants $C_q$ such that 
\begin{equ}
\E \bigl(\bigl\|\Psi_s - \Psi_t\bigr\|_{\alpha}^{2q} +  \bigl\|\PPsi_s - \PPsi_t\bigr\|_{2\alpha}^{q}\bigr) \le C_q |t-s|^{\theta q}\;,
\end{equ}
for every $q \ge 1$ and every $s,t \in [0,1]$, say. The claim then follows at once from Kolmogorov's 
continuity criterion, as for example in \cite[p.~26]{RevYor}.
\end{proof}

\begin{remark}\label{rem:canonicalLift}
Given a Gaussian process $X_t$, the area process $\XX_{s,t}$ that is canonically associated to it is 
given by
\begin{equ}
\XX_{s,t} = \lim_{\eps \to 0} \int_s^t \delta X^\eps_{s,r}\,\dot X^\eps_r\,dr\;,
\end{equ}
where $X^\eps$ is \textit{any} sequence of smooth Gaussian processes such that 
$\sup_{t \le T} \E |X^\eps_t - X_t|^2 \to 0$ as $\eps \to 0$, and satisfying a suitable uniform regularity assumption.
 It follows from the results in  \cite[Theorem~35]{GaussI}
that if the covariance of $X$ has finite two-dimensional $p$-variation for $p < 2$, then this limit exists
independently of the choice of approximating processes.
\end{remark}

We are now finally in a position to formulate what we mean exactly by a solution $u$ to \eref{e:Burgers}:

\begin{definition}\label{def:sol}
A continuous stochastic process $u$ is a solution to \eref{e:Burgers} if the process $v_t = u_t-\psi_t$ 
belongs to $\CC([0,T], \CC)\cap \L^1([0,T], \CC^1)$ and is such that the identity
\begin{equs}
\scal{v_t,\phi} &= \scal{u_0-\psi_0,\phi} + \int_0^t \scal{(\d_x^2 -1)\phi, v_s}\,ds + \int_0^t \scal{\phi, g(u_s) \d_x v_s}\,ds \label{e:defsol} \\
&\qquad + \int_0^t \int_0^{2\pi} \phi(x) g\bigl(v_s(x)+\Psi_s(x)\bigr) \,d\Psi_s(x)\,ds  + \int_0^t \scal{\phi, \hat f(u_s)}\,ds \;, 
\end{equs}  
holds almost surely for every smooth periodic test function $\phi \colon [0,2\pi] \to \R$.
Here, we have set $\hat f(u) = f(u) + u$.
\end{definition}

\begin{remark}
Since we assume that $v_s \in \CC^1$, the path $x \mapsto \phi(x) g\bigl(v_s(x) + \Psi_s(x)\bigr)$ is controlled by
$\Psi_s$, so that the inner integral on the second line is well-defined as a rough integral in the sense of \cite{Max}. Furthermore, it yields
a measurable function of $s$, being the pointwise limit of measurable functions. Its value is bounded by a constant depending on
$\|\Psi_s\|_\alpha$ and $\|\phi\|_{\CC^1}$, and depending linearly on $\|v_s\|_{\CC^1}$ (by \eref{e:boundInt} and Lemma~\ref{lem:comp}), so that the outer integral always makes sense 
as well.
\end{remark}

\begin{remark}
At first sight, one could think that this notion of solution is dependent on the arbitrary choice of 
the constant `$1$' in \eref{e:eqlin}, which in turn accounts for the presence of the function $\hat f$ in \eref{e:defsol}. 
This constant is present for the sole purpose of actually
having a stationary solution to \eref{e:eqlin}. It is however a straightforward exercise to check that the
notion of a weak solution (with the obvious modifications in \eref{e:defsol}) is independent of this choice.
\end{remark}

\subsection{Mild solutions}

It is clear that the notion of a `weak solution' given in Definition~\ref{def:sol} has a `mild solution' counterpart.
We denote by $S_t$ the heat semigroup generated by $\d_x^2$ endowed with periodic boundary conditions
and we define the heat kernel $p_t$ as being the periodic function such that
\begin{equ}
\bigl(S_t u\bigr)(x) = \int_0^{2\pi} p_t(x-y)u(y)\,dy\;,
\end{equ}
holds for every continuous function $u$.

With these notations in place, we say that $v$ is a `mild solution' to \eref{e:Burgers} if it satisfies the same
conditions as in Definition~\ref{def:sol}, but with \eref{e:defsol} replaced by the requirement that the identity
\begin{equs}
v_t(x) &= \bigl(S_t(u_0- \psi_0)\bigr)(x) + \int_0^t \bigl(S_{t-s} \bigl(g(u_s) \d_x v_s + \hat f(u_s)\bigr)\bigr)(x)\,ds \\
&\qquad + \int_0^t \int_0^{2\pi} p_{t-s}(x-y) g(u(y,s)) \,d\Psi_s(y)\,ds \;, \label{e:defvmild}
\end{equs}  
holds almost surely for every $x\in [0,2\pi]$ and $t \in (0,T]$.

Before we show that mild solutions exist, we show that (as expected) the concepts of mild and weak solutions 
do agree.

\begin{proposition}
Every mild solution is a weak solution and vice-versa.
\end{proposition}

\begin{proof}
For fixed $t$, the rough integral provides a way of interpreting
\begin{equ}
F_s(v) = g\bigl(v_s + \Psi_s\bigr) \d_x \Psi_s
\end{equ}
as an element of the space $\CS'$ of Schwartz distributions. Furthermore, it follows from
Lemma~\ref{lem:comp} that the map
$v \mapsto g\bigl(v + \Psi_s\bigr) \d_x \Psi_s$ is continuous as a map from $\CC^1$ to $\CS'$.
The claim then follows by standard techniques.
\end{proof}

From now on, we will only use the concept of a mild solution.

\subsection{Existence and uniqueness}

Our main result in this section is the following well-posedness result:

\begin{theorem}\label{theo:mild}
Let $\beta \in ({1\over 3}, {1\over 2})$ and let $u_0 \in \CC^\beta$.
Then, for almost every realisation of the driving process $\Psi$, there exists $T>0$ such that
equation \eref{e:Burgers} has a unique mild solution taking values in $\CC([0,T],\CC^\beta)$. 
If furthermore $g$ is bounded and all derivatives of $f$ and $g$ are bounded, then this solution is global
(i.e.\ one can choose $T$ arbitrary, independently of $\Psi$).
\end{theorem}

\begin{remark}
Once $\Psi\colon \R_+\to \DD^\alpha$ is fixed, our construction is completely deterministic. 
The ill-posedness of the equation \eref{e:Burgers} is then a consequence of the fact that the area process
$\PPsi_t$ is not uniquely determined by $\Psi_t$. Care needs to be taken since different numerical 
approximations to \eref{e:Burgers} may converge to solutions corresponding to different choices of the area process.
However, the canonical choice given by Lemma~\ref{lem:defPsi} is natural, as we will see in
Section~\ref{sec:IM}.
\end{remark}

Before we turn to the proof of this result, we show that:

\begin{lemma}\label{lem:convHeat}
Let  $\alpha \in ({1\over 3}, {1\over 2})$, let $\Psi \in \DD^\alpha([0,2\pi], \R^d)$ and let $(Y,Y') \in \CC^\alpha_\Psi$.
Then, there exists a constant $C$ independent of $\Psi$ and $Y$ such that 
\begin{equ}
\Bigl|\int_0^{2\pi} \d_x p_{t}(x-y) Y_y \,d\Psi(y)\Bigr| \le C t^{{\alpha\over 2}-1} \|Y\|_{\Psi,\alpha} \bigl(\|\Psi\|_\alpha +\|\PPsi\|_{2\alpha}\bigr)\;,
\end{equ} 
uniformly for $x \in [0,2\pi]$.
\end{lemma}

\begin{proof}
From the identity
\begin{equ}
\d_x p_{t}(x) = -\sum_{n \in \Z} {x\over \sqrt {2\pi} t^{3 \over 2}} \exp\Bigl({- {(x-2\pi n)^2 \over 2t}}\Bigr)\;,
\end{equ}
it is a simple exercise to check that for $t \in (0,1]$ there exist functions $f_t$ such that $\sup_{t \in (0,1]} \|f_t\|_{1,1} \le \infty$
and such that, for $x \in [-\pi,\pi]$, one has the identity
\begin{equ}
\d_x p_{t}(x) = {1\over t} f_t \Bigl({x\over \sqrt t}\Bigr)\;.
\end{equ}
The result then follows at once from Proposition~\ref{prop:convolution}.
\end{proof}

With this bound in hand, we can now turn to the

\begin{proof}[of Theorem~\ref{theo:mild}]
We perform a classical Picard iteration scheme for the fixed point equation \eref{e:defvmild}.
Fix $\alpha \in ({1\over 3}, \beta)$, and let the process
$\Psi\colon [0,1] \to \DD^\alpha([0,2\pi], \R^d)$ be given as in Lemma~\ref{lem:defPsi}.

We then consider a Picard iteration in the space of continuous functions
from $[0,T]$ to $\CC^1([0,2\pi], \R^d)$ (with $T\le 1$ to be determined), endowed with the norm
\begin{equ}
\|v\|_{1,T} = \sup_{t \le T} \|v_t\|_{\CC^1}\;.
\end{equ}
Denote this space by $\CC^1_{T}$ for the sake of conciseness. We also fix an initial condition
$u_0 \in \CC^\beta$ and we use the shorthand notation $U_t = S_{t} (u_0- \psi_0)$.

It turns out to be advantageous to subtract the contribution of the initial condition so that 
we set $v_t = u_t - \Psi_t - U_t$. With this definition, we have $v_0 = 0$ and 
we solve for the obvious modification of \eref{e:defvmild}. Note also that even though we consider
$\Psi$ as a process with values in $\DD^\alpha$, one actually has $\psi_0 \in \CC^\beta$ almost surely, and we will make
use of the additional leeway that this provides.
Given  $\Psi$, $u_0$ and $T$, we then consider the map 
\begin{equ}
\CM_{T,\Psi} \colon \CC^1_T \to \CC^1_T \;,
\end{equ}
given by
\begin{equs}
\bigl(\CM_{T,\Psi} v\bigr) (x,t) &= \int_0^t \bigl(S_{t-s} \bigl(g(u_s) \bigl(\d_x v_s + \d_x U_s\bigr) + \hat f(u_s)\bigr)\bigr)(x)\,ds \label{e:defM}\\
&\quad + \int_0^t \int_0^{2\pi} p_{t-s}(x-y) g\bigl(v(y,s) + \Psi_s(y) + U_s(y)\bigr) \,d\Psi_s(y)\,ds\\
&\eqdef \bigl(\CM_{T,\Psi}^{(1)} v\bigr) (x,t) + \bigl(\CM_{T,\Psi}^{(2)} v\bigr) (x,t)\;,
\end{equs}
where we use the shorthand notation $u_s = v_s + \psi_s + U_s$.
We now fix a realisation of $\Psi$ and we set $K > 1$ such that 
\begin{equ}
\|u_0\|_{\CC^\beta} \le K\;,\quad \|\psi_0\|_{\CC^\beta} \le K\;. 
\end{equ}
We also consider $v, \bar v$ such that
\begin{equ}
\|v\|_{1,T} < K\;,\qquad \|\bar v\|_{1,T} < K\;,
\end{equ}
and we set $\/\Psi\/ = \sup_{t \le 1} \bigl(\|\Psi_t\|_{\CC^\alpha} + \|\PPsi_t\|_{2\alpha}\bigr)$.
 
We then have a constant $c$ such that $\|u_s\|_{\infty} \le cK$ for
$s \le T$ and such that $\|U_s\|_{\CC^1} \le c K s^{\beta -1 \over 2}$.
Since furthermore $S_{t}$ is bounded by $C t^{-{1\over 2}}$ as a linear operator from $\L^\infty$ into $\CC^1$, this
immediately implies that $\CM_{T,\Psi}^{(1)} v$  belongs to $\CC^1_T$ and that
\begin{equ}[e:boundM1]
\|\CM_{T,\Psi}^{(1)} v\|_{1,T} \le C_K (1+\/\Psi\/)T^{\beta \over 2}\;,
\end{equ}
for some constant $C_K$. Note that if $g$ and $Df$ are bounded, then we can take $C_K$ proportional to $K$.

Regarding the modulus of continuity of the map $\CM_{T,\Psi}^{(1)}$, we have the identity
\begin{equs}
\bigl(\CM_{T,\Psi}^{(1)} (v- \bar v)\bigr)_t &= \int_0^t S_{t-s} \bigl(g(u_s) \bigl(\d_x v_s - \d_x \bar v_s\bigr) + \hat f(u_s) - \hat f(\bar u_s)\bigr)\,ds \\
&\quad + \int_0^t S_{t-s} \bigl(g(u_s) - g(\bar u_s)\bigr) \bigl(\d_x \bar v_s + \d_x U_s\bigr)\,ds\;.
\end{equs}
Since $\|g(u_s) - g(\bar u_s)\|_{\infty} \le C_K \|v_s - \bar v_s\|_{\infty}\le C_K \|v_s - \bar v_s\|_{\CC^1}$, and similarly for $f$, we obtain as before the bound
\begin{equ}
\|\CM_{T,\Psi}^{(1)} (v - \bar v)\|_{1,T} \le C_K (1+\/\Psi\/)\|v - \bar v\|_{1,T} T^{\beta \over 2}\;,
\end{equ}
where $C_K \propto K$ if $g$, $Dg$, and $Df$ are bounded.
Let us now turn to the second term. Here, the integrand of the inner integral should be
interpreted as a rough path controlled by $\Psi_s$, which is built from $\Psi_s$, $v_s$, and $U_s$ by
making use of Lemma~\ref{lem:comp}.

More explicitly, the integrand (without the prefactor $p_{t-s}(x-y)$) is 
the controlled rough path $(Y_s, Y_s') \in \CC^\alpha_{\Psi_s}$ given by
\begin{equs}\label{defVV'}
Y_s(x) &= g\bigl(v(x,s) + \psi(x,s) + U_s(x)\bigr)\;, \\
Y_s'(x) &= Dg\bigl(v(x,s) + \psi(x,s) + U_s(x)\bigr) \;.
\end{equs}
(We stress again that we view $s$ here simply as an index, with the `temporal' variable of our controlled 
rough path being given by $x$.) It then follows from Lemma~\ref{lem:comp} that there exists a constant
$\hat C_K$ such that 
\begin{equ}
\|Y_s\|_{\Psi_s,\alpha} \le \hat C_K \bigl(\|v_s\|_{2\alpha} + \|U_s\|_{2\alpha}\bigr)(1+\/\Psi\/)^2\;. 
\end{equ}
Since $\|U_s\|_{2\alpha} \le C s^{-{2\alpha-\beta \over 2}}$ by standard properties of the heat semigroup,
we have the bound
\begin{equ}
\|Y_s\|_{\Psi_s,\alpha} \le C_K (1+\/\Psi\/)^2 (1+ s^{-{2\alpha-\beta \over 2}})\;,
\end{equ}
Again, it is straightforward to check that if the first two derivatives of $g$ are bounded, then we can take $\hat C_K$ independent
of $K$, and therefore we have $C_K$ proportional to $K$.

It now immediately follows from Lemma~\ref{lem:convHeat} there exists a constant $C_K$ such that 
\begin{equ}
\Bigl|\d_x \int_0^{2\pi} \!\!\! p_{t-s}(x-y) Y_s(y) \,d\Psi_s(y)\Bigr| \le C_K (1+\/\Psi\/)^3 \bigl(1 + s^{-{2\alpha-\beta \over 2}}\bigr) (t-s)^{{\alpha \over 2}-1}\;,
\end{equ}
so that one has the bound
\begin{equ}[e:boundM2]
\|\CM_{T,\Psi}^{(2)} v\|_{1,T} \le C_K (1+\/\Psi\/)^3 T^{\beta - \alpha \over 2}\;,
\end{equ}
where $C_K$ is proportional to $K$ if $Dg$ and $D^2g$ are bounded.

In order to obtain control over the modulus of continuity of $\CM_{T,\Psi}^{(2)}$, we denote by 
$(\bar Y_s, \bar Y_s')$ the controlled rough path associated to $\bar v_s$, so that
\begin{equ}[e:diffY]
\bar Y_s(x) - Y_s(x) = \int_0^1 Dg\bigl(\psi_s(x) + U_s(x) +  v_s(x) + \lambda \delta v_s(x)\bigr)\bigl(\bar v_s(x)- v_s(x)\bigr)\,d\lambda\;.
\end{equ}
Applying Lemma~\ref{lem:comp} to the integrand of this expression, we obtain as before the bound
\begin{equ}
\|\bar Y_s - Y_s\|_{\Psi_s,\alpha} \le \hat C_K (1+\/\Psi\/)^2 (1+ s^{-{2\alpha-\beta \over 2}})\|\bar v_s - v_s\|_{\CC^{2\alpha}}\;,
\end{equ}
so that there exists
a constant $C_K$ such that 
\begin{equ}
\|\CM_{T,\Psi}^{(2)} v\|_{1,T} \le C_K (1+\/\Psi\/) T^{\beta - \alpha \over 2}\|v - \bar v\|_{1,T}\;.
\end{equ}
Note that even if the derivatives of $g$ are all bounded, this constant this time actually grows quadratically in $K$, 
but this turns out not to be a problem.
Combining these bounds and using the fact that $\beta > \alpha$ by assumption, 
it follows immediately that for $T$ sufficiently small, $\CM_{T,\Psi}$ maps the ball of radius
$K$ in $\CC_T^1$ into itself and satisfies $\|\CM_{T,\Psi} v\|_{1,T} \le {1\over 2}\|v - \bar v\|_{1,T}$, so that 
it admits a unique fixed point in this space. Iterating this argument in the usual way, we construct a local solution
up to some blow-up time $\tau$ with $\lim_{t \to \tau}\|u_t\|_{\CC^\beta} = +\infty$.

It remains to show that the solution constructed in this way is global if $g$, $Dg$, $D^2g$, and $Df$ are bounded.
This only uses the fact that in this case, as a consequence of \eref{e:boundM1} and \eref{e:boundM2}, 
there exists $T_\star > 0$ depending on $\/\Psi\/$ but
\textit{independent of $K$} such that 
\begin{equ}[e:boundM]
\|\CM_{T,\Psi} v\|_{1,T} \le {K\over 2}\;,
\end{equ}
for every $T \le T_\star$, provided that $\|u_0\|_{\CC^\beta} \le K$, $\|\psi_0\|_{\CC^\beta} \le K$, and $\|v\|_{1,T}\le K$.
Let now $\hat T = \inf\{t \ge 0\,:\, \|v_t\|_1 \ge K\}$. If $\hat T \le T_\star$, then the assumptions for \eref{e:boundM} to
hold are satisfied by construction so that, since $v$ is a fixed point of $\CM_{T,\Psi}$, we conclude
that $\|v_{\hat T}\|_1 \le K/2$, a contradiction. Therefore, we must have $\hat T > T_\star$, from which we conclude that 
$\tau > T_\star$. Since $T_\star$ is independent of the initial condition $u_0$, this argument can be iterated
up to arbitrarily long times, thus yielding the existence and uniqueness of global solutions.
\end{proof}

\begin{remark}
Inspection of the proof reveals that we actually only need $g \in \CC^3$ and $f \in \CC^1$ for the existence
and uniqueness of local solutions.
\end{remark}

\subsection{Stability of the solution}
\label{sec:stability}

As an almost immediate corollary of the results obtained in the previous section, we 
obtain the stability of the solutions under perturbations of the driving noise and of the initial
condition. We have the following result:

\begin{corollary}\label{cor:regular}
Let $f$ and $g$ be smooth and let $u_0, \bar u_0\in \CC^\beta$ and $\Psi, \bar \Psi \in \CC(\R_+, \DD^\alpha)\cap \CC(\R_+, \CC^\beta)$ with 
${1\over 3} < \alpha < \beta < {1\over 2}$. Denote the corresponding local solutions by $u$, $\bar u$ and the blow-up times
by $\tau$, $\bar \tau$. Then, for every such $u_0$ and $\Psi$, every $T < \tau$, and every $\eps > 0$, there exists $\delta > 0$
such that $\bar \tau \ge T$ and
\begin{equ}
\sup_{t \le T} \|u_t - \bar u_t\|_{\CC^\beta} \le \eps\;,
\end{equ}
for all $\bar u_0$ and $\bar \Psi$ such that
\begin{equ}
\Delta_{u,\Psi} \eqdef \|u_0 - \bar u_0\|_{\CC^\beta} + \sup_{t \le T} \|\Psi_t - \bar \Psi_t\|_\beta + \sup_{t \le T} \|\PPsi_t - \bar \PPsi_t\|_{2\alpha} \le \delta\;,
\end{equ}
where $\PPsi_t$ denotes the area process of $\Psi_t$ as before.
\end{corollary}

\begin{proof}
Denote by $\CM_{u_0,\Psi}$ the same map as in \eref{e:defM}, but where we change notation in order to
suppress the dependency on $T$ (which is not relevant here), and show instead the dependency on the
initial condition $u_0$. The claim then follows from standard arguments if we can show that, for every $K>0$, there
exists a constant $C_K$ such that the bound
\begin{equ}
\|\CM_{u_0,\Psi} v - \CM_{\bar u_0,\bar \Psi} v\|_{1,T} \le C_K \Delta_{u,\Psi}\;,
\end{equ}
holds provided that $\Delta_{u,\Psi} \le 1$, and that
\begin{equ}
\|u_0\|_{\CC^\beta} + \/\Psi\/ + \|v\|_{1,T} \le K\;.
\end{equ}
This in turn follows immediately from considerations similar to those in the proof of Theorem~\ref{theo:mild}.
\end{proof}

In particular, it follows from this that the notion of a solution given by Definition~\ref{def:sol} coincides with those
solutions that are obtained by molllifying the noise in \eref{e:Burgers} and passing to the limit. We can formulate this
more precisely as:

\begin{corollary}
Let $\phi\colon \R \to \R$ be a smooth compactly supported function such that $\int_\R\phi(x)\,dx = 1$ and set $\phi_\eps(x) = \phi(x/\eps)/\eps$. Define the operator
$Q_\eps\colon \L^2(S^1)\to \L^2(S^1)$ by $ \bigl(Q_\eps u\bigr)(x) = \int_{S^1} \phi_\eps(x-y)\,u(y)\,dy$ and consider the solution
$u_\eps$ to 
\begin{equ}
du_\eps = \d_x^2 u_\eps\,dt + f(u_\eps)\,dt + g(u_\eps)\,\d_x u_\eps\,dt + \sigma Q_\eps \,dW(t)\;,
\end{equ}
with blow-up time $\tau_\eps$. Then, there exists a sequence $\hat \tau_\eps$ of stopping times
converging almost surely to $\tau$ such that 
\begin{equ}
\lim_{\eps \to 0} \sup_{t \le \tau_\eps} \|u_\eps(t) - u(t)\|_{\CC^\beta}  = 0\;,
\end{equ}
where $u$ is the solution to \eref{e:Burgers} as given by Theorem~\ref{theo:mild}.
\end{corollary}

\begin{proof}
We first note that in the particular case where $\Psi_t$ is a smooth function of $x$ for every $t$ and
$\PPsi_t$ is given by \eref{e:intX} (this time reading it from right to left as a definition for $\PPsi$), then the rough integral
against $\Psi_t$ coincides with the usual Riemann integral, so that the notion of a solution given by 
Definition~\ref{def:sol} coincides with the usual notions of solution as given in \cite{DaPrato-Zabczyk92,Hairer08} 
for example.

The claim then follows from Corollary~\ref{cor:regular}, noting that $\Psi^\eps \to \Psi$ uniformly
in $\CC([0,T],\DD^\alpha)\cap\CC([0,T],\CC^\beta)$ as a consequence of \cite[Theorem~37]{GaussI}.
\end{proof}

\section{Invariant measures}
\label{sec:IM}

In this section, we show that in some cases, the invariant measure for equations of the type studied
above can be exhibited explicitly, due to the fact that the equation has a type of of gradient structure.
Indeed, let $\nu$ be the Gaussian probability measure on $\CC([0,2\pi], \R^d)$ with
covariance operator given by 
\begin{equ}[e:covnu]
K_\nu = \bigl(I - \d_x^2\bigr)^{-1}\;,
\end{equ} 
where $I$ is the identity matrix and $\d_x^2$ acts independently on every component. It is straightforward to check that 
the measure $\nu$ restricted to every subinterval of $[0,2\pi]$ of strictly smaller length is equivalent to
Wiener measure (provided that we start the Wiener process with a Gaussian initial condition). In particular, the expression
\begin{equ}
\int_0^{2\pi} G(W_x)\circ dW_x\;,
\end{equ}
is well-defined as a Stratonovich integral for every smooth function $G \colon \R^d \to \R^d$ and $\nu$-almost every $W$.
If the function $G$ has sublinear growth and a bounded derivative, then this quantity actually has exponential moments
(see Section~\ref{sec:uniformExp} below), so that we can define a probability  
measure $\mu$ as a change of measure from $\nu$ by
\begin{equ}[e:defmu]
{d\mu \over d\nu}(W) = Z^{-1}\exp \Bigl(\int_0^{2\pi} G(W_x)\circ dW_x + \int_0^{2\pi} F(W_x)\,dx\Bigr) \eqdef Z^{-1} \exp(\Xi(W))\;,
\end{equ}
where $Z$ is a normalisation constant that ensures that $\mu$ is a probability measure. Here, $F$ could be any measurable
function with subquadratic growth to ensure that this expression is integrable with respect to $\nu$.
Were it not for the periodic boundary conditions, the particular choice $F = -{1\over 2}\bigl(\div G + |G|^2\bigr)$ would yield
for $\mu$ the law of a not necessarily reversible diffusion with drift $G$.
If $G$ happens to be a gradient, the stochastic integral is reduced to a boundary term, and these equations are treated
in \cite{HairerStuartVoss07}. The main contribution of the present article is to be able to treat the non-gradient case.

Let us now suspend our disbelief for a moment and pretend just for the sake of the argument that 
$\Xi$ is differentiable as a function from $\L^2([0,2\pi],\R^d)$ into $\R$
(which of course it is not!). Formally, we can then compute the $\L^2$-derivative of $\Xi$, which yields
\begin{equ}
\bigl(D \Xi(W)\bigr)_i(x) = \d_i G_j(W(x)) {dW_j(x) \over dx} - \d_j G_i(W(x)) {dW_j(x) \over dx} + \d_i F(W(x))\;.
\end{equ}
Since on the other hand, the measure $\nu$ is invariant for the damped stochastic heat equation
\begin{equ}
du = \d_x^2 u\,dt - u\,dt + \sqrt 2 dW(t)\;, 
\end{equ}
this suggests that the measure $\mu$ given by \eref{e:defmu} is invariant for the equation
\begin{equ}[e:general]
du = \d_x^2 u\,dt + g(u)\,\d_x u\,dt + f(u)\,dt + \sqrt 2 dW(t)\;, 
\end{equ}
where $f$ and $g$ are given by
\begin{equ}[e:deffg]
f_i(u) = \d_i F(u) - u\;,\qquad g_{ij}(u) = \d_i G_j(u) - \d_j G_i(u)\;,
\end{equ}
and where $W$ is an $\L^2$-cylindrical Wiener process. (Which is just another way of stating that the driving noise
is space-time white noise, see \cite{DaPrato-Zabczyk92}.)
The aim of this section is to give a rigorous justification of this fact, which we encompass in the following theorem:

\begin{theorem}\label{theo:reversible}
Let $F$ and $G$ be $\CC^\infty$ functions with bounded derivatives of all orders such that $G$ is bounded and define
$f$ and $g$ by \eref{e:deffg}. Then, the mild solution to \eref{e:general} generates a Markov process that is reversible
with respect to the measure $\mu$ defined in \eref{e:defmu}.
\end{theorem}

We postpone the proof of this result to the end of the section 
and we first lay out the technique and prove a number of intermediate results.
Our technique will be to first consider the smoother problem with
reference measure $\nu_\eps$ having covariance operator $K_\nu^\eps$ given by
\begin{equ}[e:covnueps]
K_\nu^\eps = \bigl(I - \d_x^2 + \eps^2\d_x^4\bigr)^{-1}\;.
\end{equ}
One can check that the measure $\nu_\eps$ charges paths that are $\CC^1$. Furthermore, the map $\Xi$ is continuous
from $\CC^1$ to $\R$ (with the `stochastic integral' now being a simple Riemann integral), so that we can define
a sequence of probability measures $\mu_\eps$ by
\begin{equ}[e:defmueps]
{d\mu_\eps \over d\nu_\eps}(W) = Z_\eps^{-1} \exp(\Xi(W))\;,
\end{equ}
where $Z_\eps$ is a suitable normalisation constant. One then has the following result:

\begin{proposition}\label{prop:revueps}
Let $F$ and $G$ have bounded derivatives of all orders and let $f$ and $g$ be defined as in \eref{e:deffg}. 
Then, for every fixed value $\eps > 0$, the stochastic PDE
\begin{equ}[e:defueps]
du_\eps = -\eps^2 \d_x^4 u_\eps\,dt + \d_x^2 u_\eps\,dt  + g(u_\eps)\,\d_x u_\eps\,dt + f(u_\eps)\,dt + \sqrt 2 dW(t)\;,
\end{equ}
has unique global solutions in $\CC^1$. Furthermore, it admits the measure $\mu_\eps$ as an invariant measure
and the corresponding Markov process is reversible. 
\end{proposition}

\begin{proof}
The local well-posedness of solutions is standard and follows for example 
from \cite{Hairer08}. The global well-posedness and the invariance of 
$\mu_\eps$ then follow as in \cite[Prop.~26]{Hypo}, 
see also \cite{Zabczyk88}.
\end{proof}

Furthermore, one has the following convergence result:

\begin{proposition}\label{prop:convmu}
One has $\mu_\eps \to \mu$ weakly in the space $\CC^\alpha$ for every $\alpha < {1\over 2}$.
\end{proposition}

\begin{proof}
Denote by $\tilde \nu_\eps$ the lift of $\nu_\eps$ to a measure on $\DD^\alpha$. Since $\nu_\eps$ charges $\CC^1$ functions,
this lift is performed by simply associating to each element its `area process' given by a standard Riemann integral. 
On the other hand, we can lift $\nu$ to a measure $\tilde \nu$ on $\DD^\alpha$ in a canonical way as in Lemma~\ref{lem:defPsi}. (Note that 
this yields the same measure as if we were to construct the area process by Stratonovich integration.)
It then follows from \cite[Theorem~35]{GaussI} that $\tilde \nu_\eps \to \tilde \nu$ weakly in $\DD^\alpha$ as $\eps \to 0$.

Since  $\Xi$ is continuous as a map from $\DD^\alpha$ to $\R$, the claim then follows from the uniform exponential integrability of
$\Xi$ with respect to $\tilde \nu_\eps$. Unfortunately, this uniform exponential integrability turns out to be a highly non-trivial
fact, the proof of which is postponed to Theorem~\ref{theo:uniformExp} below.
\end{proof}

Our final ingredient is the convergence of solutions to \eref{e:defueps} to those of \eref{e:general}. 
Note that this is not a completely straightforward application of the approximation result given in Corollary~\ref{cor:regular},
since we change the linear part of the equation, rather than the noise. However, the statement is quite similar:

\begin{proposition}\label{prop:convueps}
Let $f$ and $g$ be as in the statement of Theorem~\ref{theo:reversible}, let $u_0 \in \CC^\beta$ for some $\beta \in ({1\over 3},{1\over 2})$ 
and let $\{u^\eps_0\}$ be a sequence such that $u_0^\eps \to u_0$ in $\CC^\beta$.
Then, for every $T>0$ and every $\alpha < \beta$, we have
\begin{equ}[e:boundconvu]
\lim_{\eps \to 0} \sup_{t \le T} \|u_t - u_t^\eps\|_{\CC^\alpha} = 0\;,
\end{equ}
where $u_t^\eps$ is the solution to \eref{e:defueps} and $u_t$ is the solution to \eref{e:general}, driven by the same
realisation of $W$.
\end{proposition}

\begin{proof}
It follows from the assumptions on $f$ and $g$ that both \eref{e:general} and \eref{e:defueps} have unique global 
solutions by Theorem~\ref{theo:mild}. Denote by $S^\eps_t$ the semigroup generated by the operator $\d_x^2 -1 - \eps^2 \d_x^4$,
and denote by $\psi^\eps$ the solution to the corresponding linear equation, namely
\begin{equ}
\psi^\eps(t) = \int_{-\infty}^t S^\eps_{t-s}\,dW(s)\;.
\end{equ}
We can lift this canonically to a rough-path valued process $\Psi_t^\eps$ as before. It is straightforward to show that 
the two-dimensional $p$-variation (for any $p > 1$) of the covariance of $\psi_t^\eps$ is bounded, uniformly in $\eps$, and that 
$\E |\psi^\eps(t,x) - \psi(t,x)|^2 \le C \eps$ for some $C$, uniformly in $x$ and $t$ (see for example \cite[Prop~3.10]{SemiPert}).
It then follows as before that $\Psi^\eps \to \Psi$ uniformly over bounded intervals in the rough path topology.

We set as before $v_t = u_t - \psi_t - S_t(u_0 - \psi_0)$ and $v_t^\eps = u_t^\eps - \psi_t^\eps - S_t^\eps(u_0^\eps - \psi_0^\eps)$.
It follows from Corollary~\ref{cor:approxSHold}, \cite[Prop~3.10]{SemiPert}, and standard interpolation inequalities that
there exists some constant $C$ and an exponent $\kappa$ such that 
\begin{equ}
\E \sup_{t \le T} \bigl\|\psi_t - \psi_t^\eps + S_t^\eps(u_0^\eps - \psi_0^\eps) - S_t(u_0 - \psi_0)\bigr\|_{\CC^\alpha} 
\le C \bigl(\eps^\kappa(1+\|u_0\|_{\CC^\beta}) + \|u_0^\eps - u_0\|_{\CC^\beta}\bigr)\;,
\end{equ}  
so that it suffices to show \eref{e:boundconvu} with $u_t$ and $u_t^\eps$ replaced by $v_t$ and $v_t^\eps$
respectively.

This then follows like in
the proof of Corollary~\ref{cor:regular} once we can show that
\begin{equ}[e:wanted]
\|\CM_{u_0,\Psi} v - \CM^\eps_{u_0^\eps,\Psi^\eps} v\|_{1,T} \to 0 \;,
\end{equ}
uniformly over bounded sets, where $\CM^\eps$ is the map defined like $\CM$, but with $S_t$ replaced by $S^\eps_t$. 
Since we already have a bound on $\|\CM_{u_0,\Psi} v - \CM_{u_0^\eps,\Psi^\eps} v\|_{1,T}$ from
the proof of Corollary~\ref{cor:regular}, it suffices to bound 
$\|\CM_{u_0,\Psi} v - \CM^\eps_{u_0,\Psi} v\|_{1,T}$. We break this into two terms as in \eref{e:defM}. For the first term,
we note that for $t \le T$, we have from \eref{e:reprSeps} and \eref{e:boundreg} that
\begin{equs}
\|S_t^\eps u - S_t u\|_{\CC^1} &= \|S_{t/2} \bigl(\hat S_{\eps^2 t} - 1\bigr) S_{t/2} u\|_{\CC^1}
\le Ct^{-1/2} \|\bigl(\hat S_{\eps^2 t} - 1\bigr) S_{t/2} u\|_\infty\\
&\le C\eps^{\alpha\over 2} t^{-1/2} \|S_{t/2} u\|_{\CC^\alpha}
\le C\eps^{\alpha\over 2} t^{-(1+\alpha)/2} \|u\|_\infty\;.
\end{equs}
Since this singularity is integrable, the requested bound follows.

In order to bound the term involving the rough integral, we need to perform a preliminary computation. 
Recall that we can write $\bigl(S^\eps_t u\bigr)(x) = \int p^\eps_t(x-y)\,u(y)\,dy$, where $p_t^\eps$ is given by
\begin{equ}
p_t^\eps(x) = \sum_{k \in \Z} \exp \bigl(-t - k^2 t - \eps^2 k^4 t\bigr)f_k(x)\;,
\end{equ}
where $f_k(x) = (2\pi)^{-1} e^{ikx}$. Let now $\delta p_t^\eps(x) = p_t(x) - p_t^\eps(x) $. With this notation,
it then follows from the bound $|f_k(x) - f_k(y)| \le 2 (1 \wedge k|x-y|)$ that 
\begin{equs}
|\delta p_t^\eps(x) - \delta p_t^\eps(y)| &\le 2\sum_{k \in \Z} e^{-k^2t} \bigl(1 \wedge \eps^2 k^4 t\bigr)(1 \wedge k|x-y|)\\
&\le 2\sum_{k \in \Z} e^{-k^2t} \bigl(1 \wedge \eps^2 k^4 t  \wedge k|x-y|)\;.\label{e:deltap}
\end{equs}
Note now that, by bounding the sum by an integral, one obtains the bound
\begin{equ}
\sum_{k \in \Z} e^{-k^2t} k^n \le C t^{-(n+1)/2} \;,
\end{equ}
valid for every $n > 0$. Combining this bound with \eref{e:deltap}, we obtain
\begin{equ}
|\delta p_t^\eps(x) - \delta p_t^\eps(y)| \le {C\over \sqrt t} \bigl(1 \wedge |x-y| t^{-1/2} \wedge \eps^2 t^{-1}\bigr)
\end{equ}
so that, for any pair of exponents $\alpha, \kappa \ge 0$ such that $2\alpha + \kappa \le 1$, we have
\begin{equ}[e:bounddeltap]
|\delta p_t^\eps(x) - \delta p_t^\eps(y)| \le C |x-y|^{2\alpha} \eps^{2\kappa} t^{-{1\over 2}-\alpha - \kappa}\;.
\end{equ}
It then follows from \eref{e:prodfY} and \eref{e:bounddeltap} that for $w_s = v_s + U_s \in \CC^{2\alpha}$, we have
\begin{equ}
\int_0^{2\pi} \delta p_{t-s}^\eps(x-y) g\bigl(w_s + \Psi_s(y)\bigr)\,d\Psi_s(y) \le C \eps^{2\kappa} (t-s)^{-{1\over 2}-\alpha - \kappa} (1+\|w_s\|_{\CC^{2\alpha}}) \;,
\end{equ}
where $C$ depends on the size of $\Psi$. Since $\alpha < {1\over 2}$ and the $2\alpha$-H\"older norm of $w_s = v_s + U_s$ behaves like
$s^{-\alpha + {\beta \over 2}}$, the right hand side of the above expression is integrable for every fixed $T>0$, provided that $\kappa$
is made sufficiently small.
(However the value of the integral diverges in general as $T\to 0$!)
Combining all these bounds, we conclude that \eref{e:wanted} holds, which then implies the result.
\end{proof}

It is now rather straightforward to combine all of these ingredients in order to prove Theorem~\ref{theo:reversible}:

\begin{proof}[of Theorem~\ref{theo:reversible}]
It follows from Proposition~\ref{prop:convmu} and Skorokhod's representation theorem 
\cite{Billingsley} that one can construct a sequence of random variables $u_0^\eps$ with law $\mu_\eps$ 
and a random variable $u_0$ with law $\mu$ such that $u_0^\eps \to u_0$ in $\CC^\beta$ almost surely. 

Denoting by $u_t^\eps$ the solution to \eref{e:defueps} with initial condition $u_0^\eps$ and similarly for $u_t$, 
it then follows from Proposition~\ref{prop:convueps} that $u_t^\eps \to u_t$ almost surely for every $t \ge 0$.
Since, by Proposition~\ref{prop:revueps}, the law of $u_t^\eps$ is given by $\mu_\eps$ for every $t>0$,
we conclude from Proposition~\ref{prop:convmu} that the law of $u_t$ is given by $\mu$ for all $t$.
The reversibility of $u_t$ follows in the same way from the reversibility of the $u_t^\eps$ by considering the joint
distributions at any two times.
\end{proof}

\section{Weak convergence of approximating measures}
\label{sec:uniformExp}

The aim of this section is to prove the following uniform exponential integrability result, which is essential
for the convergence result of the previous section:

\begin{theorem}\label{theo:uniformExp}
Let $G \colon \R^n \to \R^n$ be a $\CC^3$ function which is bounded, with bounded first and second derivatives. Then,
\begin{equ}
\sup_{\eps < 1} \E_\eps \exp\Bigl(\int_0^{2\pi} G(W)\,\dot W\,dt\Bigr) < \infty\;,
\end{equ}
where $\E_\eps$ is a shorthand notation for the expectation with respect to the Gaussian
measure $\nu_\eps$ with covariance given by \eref{e:covnueps}.
\end{theorem}

\begin{proof}
Since $\nu_\eps$ charges the
set of $\CC^1$ functions for every $\eps > 0$, the quantity under the expectation is 
defined for $\nu_\eps$-almost every $W$. Furthermore, the argument of the exponential is bounded by
$C \|\dot W\|_{\L^\infty}$ for some $C>0$, so that it follows from Fernique's theorem that 
\begin{equ}
\sup_{\eps \in [\eps_0,1]} \E_\eps \exp\Bigl(\int_0^{2\pi} G(W)\,\dot W\,dt\Bigr) < \infty\;,
\end{equ}
for every $\eps_0 > 0$, implying that it suffices to obtain a uniform bound for small values of $\eps$.

Note now that it follows from Cauchy-Schwarz and the translation invariance of $\nu_\eps$
that 
\begin{equ}
\E_\eps \exp\Bigl(\int_0^{2\pi} G(W)\,\dot W\,dt\Bigr) \le \E_\eps \exp\Bigl(2\int_0^{\pi} G(W)\,\dot W\,dt\Bigr)  < \infty\;,
\end{equ}
so that it suffices to bound exponential moments of the integral up to time $\pi$.

Our proof then proceeds by constructing a sequence of measures $\bar \nu_\eps$ that are equivalent to $\nu_\eps$
and such that there exists
$\alpha > 0$ and $\eps_0>0$ with 
\begin{equ}[e:boundder]
\sup_{\eps < \eps_0} \int \Bigl({d \nu_\eps \over d \bar \nu_\eps}\Bigr)^\alpha\,d \bar \nu_\eps < \infty\;.
\end{equ}
These measures will furthermore have the property that 
\begin{equ}[e:boundexp]
\sup_{\eps < \eps_0} \int \exp\Bigl(\int_0^{\pi} G(W)\,\dot W\,dt\Bigr)\,\bar \nu_\eps(dW) < \infty\;,
\end{equ}
for every function $G$ as in the statement of the result. 
The construction of $\bar \nu_\eps$ together with the uniform bound \eref{e:boundder} is the content of Lemma~\ref{lem:convHS} and Corollary~\ref{cor:convHS} below.
The uniform exponential integrability with respect to $\bar \nu_\eps$, namely \eref{e:boundexp}, is the content of Proposition~\ref{prop:unexpmu}.

Setting $\CE(W) =  \exp\bigl(\int_0^{\pi} G(W)\,\dot W\,dt\bigr)$, we now use H\"older's inequality to write
\begin{equs}
\int \CE(W)\,\nu_\eps(dW)  &= \int \CE(W) \Bigl({d \nu_\eps \over d \bar \nu_\eps}\Bigr)(W)\,\bar \nu_\eps(dW) \\
&\le \Bigl(\int \Bigl({d \nu_\eps \over d \bar \nu_\eps}\Bigr)^\alpha\,d \bar \nu_\eps\Bigr)^{1/\alpha} \Bigl(\int \CE^p(W)\,\bar \nu_\eps(dW) \Bigr)^{1/p}\;,
\end{equs}
where $p$ is the exponent conjugate to $\alpha$. The claim then follows from \eref{e:boundder} and \eref{e:boundexp} 
by noting that $\CE^p$ is again of the same form as $\CE$. The remainder of this section is devoted to the construction of
$\bar \nu_\eps$ and to the proof that \eref{e:boundder} and \eref{e:boundexp} do indeed hold.
\end{proof}

\begin{remark}
Retracing the steps of the proof, it is not difficult to check that Theorem~\ref{theo:uniformExp} still holds
if, in addition to having bounded first and second derivatives, $G$ is only assumed to have sublinear growth, 
namely $|G(u)| \le C(1+|u|)^\alpha$ for some $\alpha < 1$.
\end{remark}

\subsection{Construction of $\bar \nu_\eps$}

Essentially, we will construct $\bar \nu_\eps$ as the law of the integral of an Ornstein-Uhlenbeck process with timescale 
$\eps$, which is reflected in a suitable way around $t = \pi$. Once we know that two Gaussian measures are mutually equivalent, 
it is easy to show that a bound of the type \eref{e:boundder} holds for some $\alpha > 1$. The difficult part is to show that 
it holds uniformly in $\eps$ with the same value $\alpha$. Our main tool in this endeavour is the following standard result
from Gaussian measure theory:

\begin{proposition}\label{prop:densGauss}
Let $\bar \nu$ be a centred Gaussian measure on a separable Banach space $\CB$ with covariance operator 
$\bar C\colon \CB^* \to \CB$ and let $\nu$ be a centred Gaussian measure on $\CB$ with covariance operator $C$.
We assume that both measures have full support, so that the range of both $C$ and $\bar C$ is dense in $\CB$.

Let $\CH$ be a separable Hilbert space and let $\CA\colon \CB \to \CH$ be an unbounded operator
such that $C^{-1} = \CA^*\CA$. 
Then, $\bar \nu$ and $\nu$ are equivalent if and only if the operator
$\Lambda = 1 - \CA\bar C \CA^*$ is Hilbert-Schmidt as an operator from $\CH$ to $\CH$
and has no eigenvalue $1$.

In that case, denoting by $\{\lambda_n\}_{n \ge 0}$ the eigenvalues of $\Lambda$, one has the identity
\begin{equ}[e:dmu]
\int_\CB \Bigl({d \nu \over d\bar \nu}\Bigr)^\alpha\,d\bar \nu = \prod_{n \ge 0} {(1-\lambda_n)^{\alpha \over 2}\over \sqrt {1-\alpha \lambda_n}}\;,
\end{equ}
for all values $\alpha > 1$ such $\lambda_n \le 1/\alpha$ for all $n$.
\end{proposition}

\begin{proof}
The first statement is the content of the Feldman-H\'ajek theorem. The identity \eref{e:dmu} is straightforward to check in the case
$\CB = \R^n$ by using the fact that the left hand side is invariant under changes of coordinates, so that we can reduce ourselves
to the case $C = 1$. The general case then follows by approximation. 
\end{proof}

\begin{remark}
A canonical choice of $\CA$ and $\CH$ is to take for $\CH$ the Cameron-Martin space $\CH_\nu$ of $\nu$
and for $\CA$ the restriction operator (with domain $\CH_\nu \in \CB$). This operator can however be multiplied 
from the left by an arbitrary unitary operator without changing the statement, a fact that we will use
in the sequel.
\end{remark}

An important remark is that the right hand side of \eref{e:dmu} is continuous in the Hilbert-Schmidt topology on the 
set of operators $\Lambda$ for which the expression makes sense. As a consequence, we have the following:

\begin{corollary}\label{cor:convHS}
Let $\nu_\eps$ and $\bar \nu_\eps$ be a sequence of centred Gaussian measures
with covariances $C_\eps$ and $\bar C_\eps$ such that $C_\eps^{-1} = \CA_\eps \CA_\eps^*$ and such that the 
range of $\CA_\eps$ is some common Hilbert space $\CH$. 
Let $\Lambda_\eps = 1 - \CA_\eps \bar C_\eps \CA_\eps^*$ as before and assume that $\Lambda_\eps \to \Lambda$ in the 
Hilbert-Schmidt operator topology, where $\Lambda$ is a Hilbert-Schmidt operator with $\lambda_n < 1$ for all $n$.

Then, there exist $\alpha > 1$ and $\eps_0 > 0$ such that 
$\sup_{\eps < \eps_0} \int_\CB \bigl({d \nu_\eps \over d\bar \nu_\eps}\bigr)^\alpha\,d\bar \nu_\eps < \infty$. 
\end{corollary}

\begin{proof}
Set $1/\alpha = (1+\sup_n \lambda_n)/2 < 1$, so that the right hand side of \eref{e:dmu} is finite. The claim then follows from the
continuity of that expression in the Hilbert-Schmidt operator topology.
\end{proof}

A final remark is that as an immediate consequence of Proposition~\ref{prop:densGauss}, one has the following result,
where we identify $\CH$ with its dual in the usual way in order to consider the covariance operators as self-adjoint Hilbert-Schmidt 
operators on $\CH$:

\begin{corollary}\label{cor:diag}
Let $\nu$ and $\bar \nu$ be two centred Gaussian measures with covariance operators $C$ and $\bar C$
on a common Hilbert space $\CH$. If $C$ and $\bar C$ are simultaneously diagonalisable
with respective eigenvalues $a_n>0$ and $\bar a_n>0$, then $\nu \sim \bar \nu$ provided that
$\sum_{n \ge 1} (1- {a_n \over  \bar a_n}) < \infty$. \qed
\end{corollary}

Before we proceed with the construction of the sequence of measures $\bar \nu_\eps$, we construct their limit $\bar \nu$,
which is equivalent to the measure $\nu$ given by \eref{e:covnu}.

\begin{lemma}
Let $\nu$ be the Gaussian measure on $\L^2(S^1,\R)$ with covariance given by $(1 - \d_t^2)^{-1}$, 
and let $\bar \nu$ be the Gaussian measure with covariance 
\begin{equ}[e:defCbar]
\bar C(s,t) = 
\left\{\begin{array}{cl}
	1 + (s\wedge t) & \text{if $s\le \pi$ and $t \le \pi$,} \\
	1 + 2\pi - (s\vee t) & \text{if $s\ge \pi$ and $t \ge \pi$,} \\
	1 + 2(s\wedge t) - {st \over \pi} & \text{otherwise.} 
\end{array}\right.
\end{equ}
Then, we have $\bar \nu \sim \nu$.
\end{lemma}

\begin{proof}
In principle, the easiest way to check that the operator $\Lambda$ satisfies the assumptions of Proposition~\ref{prop:densGauss}
is to show that it is Hilbert-Schmidt and has norm strictly less than $1$. Unfortunately, it turns out that in our case, this operator
does have negative eigenvalues close to $-1$, so we use the trick of considering instead a rank one perturbation
of $\bar C$, which does gives rise to a measure which is obviously equivalent to that given by $\bar C$.
By tuning the parameter of that perturbation, we will see that $\Lambda$ can be made arbitrarily close to an operator
which is explicitly diagonalisable.

Let $\eta_\kappa$ be the Gaussian measure on $\L^2(S^1,\R)$ with covariance given by $(\kappa^2 - \d_t^2)^{-1}$, 
and let $\bar \eta_\kappa$ be the Gaussian measure with covariance 
\begin{equ}
\bar C_\kappa(s,t) = 
\left\{\begin{array}{cl}
	\kappa^{-2} + (s\wedge t) & \text{if $s\le \pi$ and $t \le \pi$,} \\
	\kappa^{-2} + 2\pi - (s\vee t) & \text{if $s\ge \pi$ and $t \ge \pi$,} \\
	{\kappa^{-2}}  + 2(s\wedge t) - {st \over \pi} & \text{otherwise.} 
\end{array}\right.
\end{equ}
Then, we will show that $\eta_\kappa \sim \bar \eta_\kappa$ for $\kappa$ sufficiently large.
Since one obviously has $\bar \eta_\kappa \sim \eta_{\kappa'}$ for every $\kappa, \kappa' > 0$ (the two covariance operators
differ by a rank $1$ perturbation) and
since it follows easily from Corollary~\ref{cor:diag} that $\eta_\kappa \sim \eta_{\kappa'}$, the claim then follows.

Setting $\CB = \CH = \L^2(S^1,\R)$ and $\CA = \kappa - \d_t$, we see that $\CA$ has the required property since $\CA^*\CA = \kappa^2 - \d_t^2$.
Furthermore, the corresponding operator $\Lambda_\kappa = 1 - \CA \bar C_\kappa \CA^\star$ is an integral operator with kernel given in the sense of distributions by 
\begin{equ}
\Lambda_\kappa(s,t) = \delta(t-s) - (\kappa - \d_t)(\kappa - \d_s) \bar C_\kappa (s,t)\;.
\end{equ}
An explicit calculation shows that 
\begin{equ}
\Lambda_\kappa(s,t) = \Lambda_0(s,t) + \CO(\kappa)\;,
\end{equ}
where $\CO(\kappa)$ means that the remainder term is uniformly bounded by some constant times $\kappa$.
The kernel $\Lambda_0$ is given by
\begin{equ}
\Lambda_0(s,t) = 
\left\{\begin{array}{cl}
	-\pi^{-1} & \text{if $(s,t) \in [0,\pi]^2 \cup [\pi,2\pi]^2$,} \\
	0 & \text{otherwise.} 
\end{array}\right.
\end{equ}
It is straightforward to check that the operator $\Lambda_0$ (we identify an operator with 
the corresponding integral kernel) is negative definite, since it can be diagonalised by
considering linear combinations of step functions.
The claim then follows at once, since the Hilbert-Schmidt norm of the remainder term is of order $\kappa$.
\end{proof}

\begin{remark}
The covariance $\bar C$ is realised by the following construction. Take a standard Wiener process $W$,
 an independent Brownian bridge $B$, and an independent normal random variable $\xi$. 
 Then,
the process $X$ defined by
 \begin{equ}
X_t = \xi + W_t\;,\quad t \in [0,\pi]\;,\qquad X_t = \xi + (2\pi-t) W_\pi + B_{2\pi-t}\;,\quad t\in [\pi,2\pi]\;,
\end{equ}
does have $\bar C$ as its covariance operator.
\end{remark}

We now make the following construction for $\eps > 0$.
Let $Z^\eps$ be a stationary Ornstein-Uhlenbeck process with characteristic time $\eps$ and variance 
${1\over {2\eps}}$, so that its covariance is given by
\begin{equ}
\E Z^\eps_s Z^\eps_t = {1\over 2\eps} \exp\Bigl(-{|t-s|\over \eps}\Bigr)\;.
\end{equ}
We then define $X^\eps$ as the integral of $Z^\eps$, so that $X^\eps_t = \int_0^t Z^\eps_s\,ds$.
The covariance of $X^\eps$ is then given by
\begin{equ}[e:covX]
\E X^\eps_s X^\eps_t = (s\wedge t) - {\eps \over 2} \bigl(1 + e^{-|t-s|/\eps} - e^{-s/\eps} - e^{-t/\eps}\bigr) \eqdef K^\eps(s,t)\;.
\end{equ}
Denote now by $A^\eps$ the vector consisting of the ``boundary data'' 
$(\sqrt {2\eps} Z^\eps_0, \sqrt {2\eps}Z^\eps_\pi, X^\eps_\pi)$, so that the covariance of $A^\eps$ is given by
\begin{equ}
Q_\eps = 
\begin{pmatrix}
1 & \delta  & \sqrt {\eps \over 2} (1- \delta) \cr
\delta  & 1 & \sqrt {\eps \over 2} (1- \delta) \cr
\sqrt {\eps \over 2} (1- \delta)&\sqrt  {\eps \over 2} (1- \delta) & \pi -  \eps \bigl(1- \delta \bigr)
\end{pmatrix}
\;,\quad \delta = e^{-\pi /\eps}\;.
\end{equ}
Here, we have normalised $Z^\eps$ in such a way that it is of order one.
We also note that the covariance of $X^\eps_t$ with $A^\eps$ is given by
\begin{equ}[e:vt]
v^\eps_t := \E X^\eps_t A^\eps = \Bigl(\sqrt{\eps \over 2}(1-e^{-t/\eps}), \sqrt{\eps \over 2}(e^{-(\pi-t)/\eps} - \delta),   K^\eps(t,\pi) \Bigr)\;,
\end{equ}
so that there exists a process $Y^\eps$ independent of $A^\eps$ such that
\begin{equ}[e:exprX]
X^\eps_t = Y^\eps_t + \scal{Q_\eps^{-1} v^\eps_t, A^\eps}\;.
\end{equ}
The covariance of the process $Y^\eps_t$ is then given by
\begin{equ}
\E Y^\eps_s Y^\eps_t = K^\eps(s,t) - \scal{Q_\eps^{-1} v^\eps_t, v^\eps_s}\;.
\end{equ}
Let now $\tilde Y^\eps$ be a process independent of $Y^\eps$ and $A^\eps$, but identical in law
to $Y^\eps$, and define $\tilde X^\eps$ to be the stochastic process given by
\begin{equ}[e:defXtilde]
\tilde X^\eps_t = \tilde Y^\eps_t + \scal{Q_\eps^{-1} v^\eps_t, J A^\eps}\;,
\end{equ}
where $J$ is the matrix given by $J = \diag(-1,-1,1)$. Note that $\tilde X^\eps$ is \textit{not} 
identical in law to $X^\eps$ because, while we force the identity $\tilde X^\eps(\pi) = X^\eps(\pi)$, we force 
the derivatives at the boundary points to satisfy the relations
\begin{equ}[e:gluing]
\d_t \tilde X^\eps(t) = - \d_t X^\eps(t)\;,\qquad t \in \{0,\pi\}\;.
\end{equ}
Actually, it follows from \eref{e:exprX} and \eref{e:defXtilde} that the covariance of $\tilde X$ can be written as
\begin{equ}[e:covXtilde]
\E \tilde X^\eps_s \tilde X^\eps_t  = \E X^\eps_s X^\eps_t  + \scal[b]{Q_\eps^{-1} v^\eps_t, \bigl(JQ_\eps J - Q_\eps\bigr)Q_\eps^{-1} v^\eps_s}\;.
\end{equ}
With all of these notations at hand, we let $\xi$ be a $\CN(0,1)$ distributed random variable
independent of $X^\eps$ and $\tilde X^\eps$, and we define a process $W^\eps$ by
\begin{equ}
W^\eps_t = 
\left\{\begin{array}{cl}
	\xi + X^\eps_t & \text{if $t \le \pi$,} \\
	\xi + \tilde X^\eps_{2\pi-t} & \text{otherwise.}
\end{array}\right.
\end{equ}
We denote by $\bar C_\eps$ the covariance of $W^\eps$ and we let $\bar \nu_\eps$ be the law
of $W^\eps$. Our aim now is to show that the sequence $(\bar \nu_\eps, \nu_\eps)$,
where $\nu_\eps$ has covariance operator $(1-\d_t^2 + \eps^2 \d_t^4)^{-1}$, satisfies the assumptions of Corollary~\ref{cor:convHS}.
To this end, we note that we can write
\begin{equ}
1-\d_t^2 + \eps^2 \d_t^4 = \CA_{\eps}\CA_{\eps}^*\;,\qquad
\CA_{\eps} = 1 + \sqrt{1-2\eps}\,\d_t - \eps\d_t^2\;,
\end{equ}
so that the integral operator $\Lambda^\eps = 1-\CA_\eps \bar C_\eps \CA_\eps^*$ associated to the pair
$(\nu_\eps, \bar \nu_\eps)$ is given by
the integral kernel
\begin{equ}
\Lambda^\eps(s,t) = \delta(t-s) - \CA_{\eps;t}\CA_{\eps;s} \bar C_\eps(s,t)\;,
\end{equ}
where we denote by $\CA_{\eps;s}$ the differential operator $\CA_\eps$ acting on the $s$ variable and
similarly for $t$. Recall that on the other hand, the operator $\Lambda$ associated in the same way to $(\nu, \bar \nu)$
is given by
\begin{equ}
\Lambda(s,t) = \delta(t-s) - \CA_{0;t}\CA_{0;s} \bar C (s,t)\;,
\end{equ}
with $\bar C$ defined as in \eref{e:defCbar}. 
With these notations at hand, we then have

\begin{lemma}\label{lem:convHS}
We have $\|\Lambda - \Lambda^\eps\| = \CO(\sqrt \eps)$, where $\|\cdot\|$ denotes the $\L^2$-norm on $[0,2\pi]^2$,
which is identical to the Hilbert-Schmidt norm when identifying kernels with the corresponding integral operators.
\end{lemma}

\begin{proof}
Note first that $\bar C_\eps$, viewed as a $2\pi$-periodic function in both arguments, 
is smooth everywhere, except at $s=t$ and at $s, t \in \{0,\pi\}$. We first analyse the singularities and
then proceed to the proof of the bound on $\Lambda - \Lambda^\eps$.

We argue that the singularities at $\{0,\pi\}$ are harmless. Indeed, it follows from the gluing condition \eref{e:gluing}
that $\bar C_\eps$ is globally $\CC^1$, with possible jump-discontinuities in its second derivatives. Considering the
singularity $s=\pi$ for example, it then follows that  
$\CA_{\eps;t}\bar C_\eps$ is still $\CC^1$ in the vicinity of this discontinuity line, so that $\CA_{\eps;s}\CA_{\eps;t}\bar C_\eps$
has at most a jump discontinuity. 

Let us now turn to the singularity at $s=t$. It follows immediately from \eref{e:covX} that
one has 
\begin{equ}
\bar C_\eps(s,t) = {|t-s|^3 \over 12 \eps^2} + R_\eps^1(s,t)\;,
\end{equ}
where $R_\eps^1$ is $\CC^4$  in a vicinity of this singularity. It follows immediately that
\begin{equ}
\CA_{\eps;t}\bar C_\eps(s,t) = -{|t-s| \over 2 \eps} + R_\eps^2(s,t)\;,
\end{equ}
where $R_\eps^2$ is $\CC^1$ with a jump discontinuity in its second derivative in a vicinity of the singularity.
It follows from these considerations that we do indeed have $\CA_{\eps;s}\CA_{\eps;t}\bar C_\eps = \delta(t-s)$,
up to a smooth function with jump discontinuities along the singularity lines, so that it suffices to bound $\Lambda - \Lambda^\eps$
away from the singularities.

In order to do so, we first introduce the vector-valued function $\Psi$ given by
\begin{equ}
\Psi_t = \bigl(1, t, \sqrt \eps e^{-t/\eps}, \sqrt \eps e^{-(\pi-t)/\eps}\bigr)\;.
\end{equ}
With this notation, it follows from \eref{e:covX}, \eref{e:covXtilde}, and \eref{e:vt} that there exist
matrices $\CK^\eps$ and $\CK_{\star\!\star}^\eps$ such that, for $s,t \in [0,\pi]$,
\begin{equs}
\E X_t X_s &= (s \wedge t) - {\eps \over 2} e^{-|t-s|/\eps} + \sqrt \eps\scal{\Psi_t, \CK^\eps\Psi_s}\;, \\
\E \tilde X_t \tilde X_s &= (s \wedge t) - {\eps \over 2} e^{-|t-s|/\eps} + \sqrt \eps\scal{\Psi_t,\CK^\eps_{\star\!\star}\Psi_s}\;,
\end{equs}
where $\CK^\eps$ and $\CK^\eps_{\star\!\star}$ both have all of their entries bounded by some constant independent of $\eps$.
For $X$, this is obvious by inspection, and for $\tilde X$ it follows from the fact that $JQ_\eps J = Q_\eps + \CO(\sqrt \eps)$.

In order to obtain a similar expression for $\E X_t \tilde X_s$, we note from \eref{e:vt} that there exists a vector
$V^\eps$ with all entries bounded by a constant independent of $\eps$ such that 
\begin{equ}
v_t^\eps =  \Bigl(-\sqrt{\eps \over 2}e^{-t/\eps}, \sqrt{\eps \over 2}e^{-(\pi-t)/\eps},  t\Bigr) + \sqrt \eps \scal{V^\eps, \Psi_t}\;.
\end{equ}  
Since furthermore $JQ_\eps^{-1} = \diag(-1,-1,\pi^{-1}) + \CO(\sqrt \eps)$, we conclude that 
there exist a matrix $\CK_\star^\eps$ with uniformly bounded entries such that
\begin{equ}
\E X_t \tilde X_s = \scal{v_t^\eps, JQ_\eps^{-1} v_s^\eps} = {st \over \pi} - {\eps \over 2} \bigl(e^{-{t+s\over \eps}} + e^{-{2\pi - t - s \over \eps}}\bigr)
 + \sqrt \eps\scal{\Psi_t, \CK_\star^\eps\Psi_s}\;.
\end{equ}
Combining these expressions, we conclude that there exists a matrix $\bar \CK^\eps$ uniformly bounded in $\eps$
such that 
\begin{equ}
\bar C_\eps(s,t) = \bar C(s,t) - {\eps \over 2} \bigl(e^{-|t-s|/\eps} +e^{-|t-s+2\pi|/\eps} +e^{-|t-s-2\pi|/\eps}\bigr) + \sqrt \eps \scal{\bar \Psi_t, \bar \CK^\eps \bar \Psi_s}\;,
\end{equ}
where $\bar \Psi_t$ consists of $\Psi_t$ restricted to $[0,\pi]$, as well as its translate to $[\pi, 2\pi]$.
Writing $\CA_\eps = \CA_0 + \delta \CA_\eps$, we conclude that 
\begin{equs}\label{e:boundLambda}
\Lambda - \Lambda_ \eps &= \bigl(\CA_{0;t}\delta \CA_{\eps;s} + \CA_{0;t}\delta \CA_{\eps;s} + \delta \CA_{\eps;t}\delta \CA_{\eps;s}\bigr)\bar C(s,t)\\
&\quad - {\eps \over 2} \CA_{\eps;s}\CA_{\eps;t}\bigl(e^{-|t-s|/\eps} +e^{-|t-s+2\pi|/\eps} +e^{-|t-s-2\pi|/\eps}\bigr)\\
&\quad + \sqrt \eps \scal{(\CA_{\eps}\bar \Psi)_t, \bar \CK^\eps (\CA_{\eps}\bar \Psi)_s}\;.
\end{equs}
Since $\bar C$ is smooth outside of the singular set and since $\delta \CA_\eps$ is a differential operator with coefficients of order $\eps$,
the first term is bounded by $\CO(\eps)$ uniformly in $s,t$. In order to bound the third term, we can check by inspection that the
$\L^2$-norm of $\CA_\eps \Psi$ is bounded, uniformly in $\eps$ (this is how the scaling of the various terms was chosen in the first place),
so that this term has $\L^2$-norm bounded by $\CO(\sqrt \eps)$. It therefore remains to bound the terms appearing on the second line
in \eref{e:boundLambda}. Since these terms are just translates of each other, it is sufficient to bound the first one and, by symmetry,
it suffices to consider the region $s \le t$, so that the term in question is given by a constant multiple of 
\begin{equs}
\eps \CA_{\eps;s} e^{s/\eps} \CA_{\eps;t}e^{-t/\eps} &= \eps \Bigl(1 + {\sqrt{1-2\eps} \over \eps} - {1\over \eps}\Bigr)\Bigl(1 - {\sqrt{1-2\eps} \over \eps} - {1\over \eps}\Bigr) e^{(s-t)/\eps}\\
&= \eps e^{(s-t)/\eps}\;.
\end{equs}
This term has an $\L^2$-norm of order $\CO(\eps^{3/2})$, so that the claim follows.
\end{proof}

\subsection{Exponential moment bounds for $\mu_\eps$}

Let $Z^\eps$ be an $n$-dimensional Ornstein-Uhlenbeck process with time-scale $\eps$, that is
\begin{equ}
Z^\eps_t = \int_{-\infty}^t f_\eps(t-s)\,dW(s)\;,\qquad f_\eps(t) = \eps^{-1}e^{-t/\eps}\;,
\end{equ}
where $W$ is a standard $n$-dimensional Wiener process.
Note that $f_\eps$ is an approximation to the Dirac $\delta$-function, so that for small values of $\eps$, $Z^\eps$
is an approximation to white noise. We will also use the shorthand notation $F_\eps = 1- \exp(-t/\eps)$ in the sequel,
so that $F_\eps(t) = \int_0^t f_\eps(s)\,ds$.
We then define
\begin{equ}
W^\eps_t = \int_0^{t} Z^\eps_s\,ds\;,
\end{equ}
which can also be rewritten as
\begin{equ}
W^\eps_t = \int_{-\infty}^0 \bigl(F_\eps(t-s) - F_\eps(-s)\bigr) dW(s) + \int_0^{t} F_\eps(t-s)\,dW(s)\;.
\end{equ}
Note that for small values of $\eps$, this is a good approximation to Brownian motion. 
In particular, the variance of any component of $W^\eps_t$ is given by
\begin{equ}[e:covWeps]
\E |W^{\eps,i}_t|^2 = t - \eps F_\eps(t)\;.
\end{equ}
The aim of this section is to show that 
if $G\colon \R^n \to \R^n$ is a sufficiently nice function, then the quantity
\begin{equ}[e:defHeps]
H_\eps := \int_0^{\pi} \scal{G(W^\eps_t), Z^\eps_t}\,dt\;,
\end{equ}
is uniformly exponentially integrable as $\eps \to 0$. 
Indeed, we have the following result:

\begin{proposition}\label{prop:unexpmu}
Let $G\colon \R^n\to \R^n$ be a bounded $\CC^2$ function with bounded first and second derivatives.
Then, there exists $\eps_0 > 0$ such that 
\begin{equ}
\sup_{\eps < \eps_0} \E \exp  \int_0^{\pi} \scal{G(W^\eps_t), Z^\eps_t}\,dt < \infty\;.
\end{equ}
\end{proposition}

\begin{proof}
Let $H_\eps$ be as in \eref{e:defHeps}. It follows from the Clark-Ocone formula \cite{Nualart} that
we can represent $H_\eps$ as a stochastic integral:
\begin{equ}[e:reprHeps]
H_\eps = \E H_\eps + \int_{-\infty}^{\pi} \E \bigl(\D_s^j H_\eps\,|\, \F_s\bigr)\,dW^j(s)\;,
\end{equ}
 (here and in the sequel, summation over repeated indices is implicit) 
 where $\D_s^j$ denotes the Malliavin derivative at time $s$ with respect to $W^j$
and $\F_s$ is the filtration generated by the increments of $W$. Here, the convention we use in the definition
of $\D_s$ is such that if $f$ is any smooth deterministic function, then
\begin{equ}
\D_s^j \int_0^\pi \scal{f(s)\,dW(s)} = f^j(s)\;,\qquad s \in [0,\pi]\;.
\end{equ}
Since the second term in \eref{e:reprHeps} is a martingale, 
the result then follows if we can show that $\E H_\eps$ is uniformly bounded as $\eps \to 0$ and that the quantity
\begin{equ}[e:finalbound]
J_\eps^j \eqdef \int_{-\infty}^{\pi} \bigl|\E \bigl(\D_s^j H_\eps\,|\, \F_s\bigr)\bigr|^2\,ds
\end{equ}
is uniformly exponentially integrable for every $G$ satisfying the assumptions of the theorem.

The expectation of $H_\eps$ is 
relatively straightforward to bound. Indeed, we have the identity
\begin{equ}
\E \bigl(Z^\eps_t \,|\, W^\eps_t \bigr) =  W^\eps_t {\E Z^\eps_t W^\eps_t \over \E |W^\eps_t|^2} = {W^\eps_t \over 2} {F_\eps(t)\over t - \eps F_\eps(t)}\;,
\end{equ}
which implies that
\begin{equ}
|\E H_\eps| \le {1\over2}\int_0^{\pi}{F_\eps(t) \over t - \eps F_\eps(t)} \E \bigl(\scal{G(W^\eps_t), W^\eps_t}\bigr)\,dt
\le C \int_0^{\pi}{F_\eps(t) \over \sqrt{t - \eps F_\eps(t)}}\,dt\;,
\end{equ}
where we have used the boundedness of $G$
together with \eref{e:covWeps} to obtain the second bound. 
Since $F_\eps(t) \approx {t\over \eps} - {t^2 \over 2\eps^2}$ for $t \ll \eps$ and $F_\eps(t) \approx 1$ for $t \gg \eps$,
one can check that the integrand is uniformly bounded by $C/\sqrt t$ for some $C>0$, so that $|\E H_\eps|$ is indeed
bounded by a constant independently of the value of $\eps \le 1$. 

Let us now turn to the bounds on $\D_s H_\eps$. It follows immediately from the definition of the Malliavin derivative
that one has the identities
\begin{equ}
\D_s^i W^{\eps,j}_t = \delta_{ij} \bigl(F_\eps(t-s)\one_{s \le t} - F_\eps(-s)\one_{s \le 0}\bigr)\;,\quad \D_s^i Z^{\eps,j}_t = \delta_{ij} f_\eps(t-s)\one_{s \le t}\;.
\end{equ}
We treat the case $s < 0$ separately from the case $s \ge 0$. For $s \ge 0$, we have
\begin{equs}
 \E \bigl(\D_s^j H_\eps\,|\, \F_s\bigr) &= \int_s^{\pi} \E\bigl( \d_j G_i (W^\eps_t) F_\eps(t-s)Z_\eps^i(t) + G_j(W^\eps_t) f_\eps(t-s)\,\big|\, \F_s\bigr)\,dt\\
 &\eqdef J_1^j(s) + J_2^j(s)\;.
\end{equs}
The term $J_2^j$ is easy to bound since $G$ is bounded and the variation of $f_\eps$ is also bounded, uniformly in $\eps$, so that 
\begin{equ}[e:boundJ2]
|J_2(s)| \le \|G\|_{\infty} \int_s^{\pi} f_\eps(t-s)\,dt \le \|G\|_{\infty}\;,
\end{equ}
holds almost surely.
The first term is much more tricky to bound. We write 
\begin{equ}[e:exprJ1]
J_1^j(s) = \E\Bigl(\int_s^{2\pi}  \d_j G_i (W^\eps_t) F_\eps(t-s)\,\E\bigl(Z^{\eps,i}_t\,\big|\, \F_s \vee W^\eps_t\bigr)\,dt\,\Big|\,\F_s\Bigr)\;,
\end{equ}
and we note that we have the identities
\begin{equs}
Z^\eps_t &= \d_t a_s(t) + \int_s^t f_\eps(t-r)\,dW(r)\;,\\
W^\eps_t &= a_s(t) + \int_s^t F_\eps(t-r)\,dW(r)\;,
\end{equs}
where we introduced the shorthand notation
\begin{equ}
a_{s}(t) = W^\eps_s + \eps F_\eps(t-s)Z^\eps_s\;.
\end{equ}
This implies that
\begin{equs}[e:condZeps]
\E\bigl(Z^\eps_t\,\big|\, \F_s \vee W^\eps_t\bigr) &=  \d_t a_s(t) + M_\eps(t-s)\bigl(W^\eps_t - a_s(t) \bigr)\;,\\
M_\eps(t) &= {F_\eps^2(t) \over 2t - \eps F_\eps(t)(2+F_\eps(t))}\;.
\end{equs}
An important fact is that the function $M_\eps$ has the property that there exists a constant $C$ such that the bound
\begin{equ}[e:upperboundM]
|M_\eps(t)| \le {C\over t}\;,
\end{equ}
holds for every $t \ge 0$, uniformly in $\eps > 0$. Note that the scaling properties of $F_\eps$ imply that 
$M_\eps(t) = \eps^{-1} M_1(t/\eps)$, so that this needs to be shown only for $M_1$.
On the one hand, \eref{e:upperboundM} clearly holds for $t \to \infty$. On the other hand, one can check that
$\lim_{t \to 0} tM_\eps(t) = {3\over 2}$. The claim then follows by noting that the numerator  in the definition of $M_\eps$
is bounded and that the denominator is an increasing function of $t$.

We now introduce the further shorthand notation
\begin{equ}
T_{\eps}(t) = \int_0^{t} F_\eps^2(r)\,dr\;, 
\end{equ}
and we denote by $\CP_T$ the heat semigroup on $\R^n$. With these shorthands, we can combine \eref{e:condZeps} 
and \eref{e:exprJ1} to obtain the following explicit expression for $J_1$:
\begin{equs}
J_1^j(s) &= \int_s^{\pi} {\d_t a_s^i(t)} \bigl(\CP_{T_\eps(t-s)} \d_j G_i\bigr)(a_s(t))\,F_\eps(t-s)\,dt \\
&\quad  + {1\over 2}\int_s^{\pi} M_\eps(t-s)\sqrt {T_\eps(t-s)} \bigl(\CP_{T_\eps(t-s)} \d_{ij}^2 G_i\bigr)(a_s(t))\,F_\eps(t-s)\,dt\\
&\eqdef J_{1,1}^j(s) + J_{1,2}^j(s)\;.
\end{equs}
Since $\d_j G_i$ is bounded by assumption, it follows immediately from the definition of $J_{1,1}^j(s)$ that
there exists a constant $C$ such that the bound
\begin{equ}[e:boundJ11]
|J_{1,1}^j(s)| \le  \eps |Z^\eps_s|\,\|DG\|_{\infty}\;,
\end{equ}
holds almost surely. In order to bound $J_{1,2}^j$, we note that $M_\eps(t-s)\sqrt {T_\eps(t-s)} \le {3\over 2}|t-s|^{-1/2}$
so that the bound
\begin{equ}[e:boundJ12]
|J_{1,2}^j(s)| \le C\int_s^{\pi} {\bigl|\bigl(\CP_{T_\eps(t-s)} \d_{ij}^2 G_i\bigr)(a_s(t))\Bigr| \over \sqrt{t-s}} \,dt
\le  C \sqrt{\pi - s}\,\|D^2G\|_{\infty}\;,
\end{equ}
holds almost surely.

We now turn to the case $s < 0$. For this case, we have the identity
\begin{equs}
 \E \bigl(\D_s^j H_\eps\,|\, \F_s\bigr) &= \int_0^{\pi} \E\bigl( \d_j G_i (W^\eps_t) \bigl(F_\eps(t-s) \\
 &\qquad - F_\eps(-s)\bigr) Z_\eps^i(t) + G_j(W^\eps_t) f_\eps(t-s)\,\big|\, \F_s\bigr)\,dt\\
 &\eqdef J_3^j(s) + J_4^j(s)\;.
\end{equs}
The term $J_4^j$ is easy to bound as before, yielding
\begin{equ}[e:boundJ4]
|J_4^j(s)| \le e^{s/\eps} \|G\|_{\infty} \;,
\end{equ}
so that 
\begin{equ}
\int_{-\infty}^0 |J_4^j(s)|^2 \le {\eps\over 2} \|G\|_{\infty}^2\;.
\end{equ}
In order to bound $J_3^j$, we note that we have the identities
\begin{equ}
Z^\eps_t = e^{-(t-s)/\eps}Z^\eps_s + \int_s^t f_\eps(t-r)\,dW(r)\;,
\end{equ}
as well as $F_\eps(t-s) - F_\eps(-s) = e^{s/\eps}F_\eps(t)$. It follows that 
\begin{equs}
|J_3^j| &\le e^{2s/\eps} \|\d_j G\|_{\infty} \Bigl(\eps |Z_s^\eps| + \int_0^\pi \sqrt{\E \Bigl(\Bigl(\int_s^t f_\eps(t-r)\,dW(r)\Bigr)^2\,\Big|\,\F_s\Bigr)}\,dt\Bigr) \\
&\le e^{2s/\eps} \|\d_j G\|_{\infty} \bigl(\eps |Z_s^\eps| + \eps^{-1/2}\bigr) \eqdef e^{2s/\eps} \tilde J_3^j(s)\;.
\end{equs}
In order to conclude, note that by \eref{e:finalbound} it remains to show that $K\int_{0}^\pi |J_k^j(s)|^2\,ds$ is uniformly
exponentially integrable for every $K>0$ and for $k \in \{1,2\}$, and similarly for $K\int_{-\infty}^0 |J_k^j(s)|^2\,ds$
with $k \in \{3,4\}$. These bounds are trivial for $k \in \{2,4\}$ by \eref{e:boundJ2} and \eref{e:boundJ4}.
To bound the term involving $J_1^j$, we deduce from Jensen's inequality that
\begin{equ}
\E \exp \Bigl(K\int_0^\pi |J_1^j(s)|^2\,ds\Bigr) \le {1\over \pi} \int_0^\pi \E \exp \bigl(\pi K |J_1^j(s)|^2\bigr)\,ds\;.
\end{equ}
Since $Z_s^\eps$ is Gaussian with variance $\eps^{-1}$, we deduce that for any $K>0$ we can choose $\eps_0$ small 
enough so that the bound
$\E \exp(\eps^2 K |Z_s^\eps|^2) < 2$ holds for every $\eps < \eps_0$. The requested bound for $J_1$ thus follows from
\eref{e:boundJ11} and \eref{e:boundJ12}. 

It remains to bound the term involving $J_3^j$. We similarly obtain from Jensen's inequality that
\begin{equ}
\E \exp \Bigl(K\int_{-\infty}^0 e^{4s/\eps} |\tilde J_3^j(s)|^2\,ds\Bigr)
\le {4 \over \eps} \int_{-\infty}^0 e^{4s/\eps} \E \exp \Bigl({K\eps \over 4}  |\tilde J_3^j(s)|^2 \Bigr)\,ds\;.
\end{equ}
The requested bound then follows as before, using the explicit form of $\tilde J_3$.
\end{proof}

\appendix
\section{Semigroup bounds}

In this appendix, we collect some elementary results on the way that the semigroup $S^\eps_t$ introduced in
Section~\ref{sec:IM} approximates the damped heat semigroup $S_t$. To investigate this, we consider the semigroup
$\hat S_t$ generated by $-\d_x^4$ (always with periodic boundary conditions), and we use the fact that 
one has the identity
\begin{equ}[e:reprSeps]
S^\eps_t = \hat S_{\eps^2 t} S_t\;.
\end{equ}
The semigroup $\hat S$ can be described explicitly with the help of the kernel $\phi$ defined by
\begin{equ}
\phi(x) = {1\over 2\pi}\int_\R e^{-k^4-ikx}\,dk\;.
\end{equ}
With this definition, one has the identity
\begin{equ}[e:reprhatS]
\bigl(\hat S_t u\bigr)(x) = t^{-1/4}\int_\R \phi \bigl(y t^{-1/4}\bigr) u(x+y)\,dy\;,
\end{equ}
where we identify $u$ with its periodic continuation.
As a consequence of this representation, one has:

\begin{proposition}\label{prop:SC}
One has, $\sup_{t \le 1} \|\hat S_t\|_{\CC^\alpha \to \CC^\alpha} < \infty$.
\end{proposition}

\begin{proof}
Since convolution with a periodic function $\psi$  is an operation on $\CC^\alpha$ with norm bounded
by the $\L^1$-norm of $\psi$, the claim now follows from \eref{e:reprhatS}, using the fact that $\phi$ is 
in $\L^1$ and that the scaling by $t^{-1/4}$ does not change its $\L^1$ norm.
\end{proof}

It follows that one has the following approximation result:

\begin{corollary}\label{cor:approxSHold}
For every  $\beta \in (0,1)$, every $T>0$, and every $\alpha \in (0,\beta)$, 
there exists a constant $C>0$ such that
\begin{equ}
\|S^\eps_t u - S_t u\|_{\CC^\alpha} \le C \eps^{{\beta - \alpha\over 2}}\;,
\end{equ}
uniformly for all $t \in [0,T]$ and all $u \in \CC^\beta$.
\end{corollary}

\begin{proof}
We note that it follows from \eref{e:reprhatS} that
\begin{equ}[e:boundreg]
\bigl|\bigl(\hat S_\eps u - u\bigr)(x)\bigr| \le \eps^{\beta/4} \|u\|_\beta \int_\R \eps^{-1/4} |\phi(x \eps^{-1/4})| |x \eps^{-1/4}|^\beta\,dx\le C \eps^{\beta/4} \|u\|_\beta\;.
\end{equ}
On the other hand, we know from Proposition~\ref{prop:SC} that the bound
\begin{equ}
\|\hat S_\eps u - u\|_{\CC^\beta} \le C \|u\|_{\CC^\beta} \;,
\end{equ}
is valid for $\eps < 1$, say. Combining these bounds, we conclude that
\begin{equ}
\|\hat S_\eps u - u\|_{\CC^\alpha} \le C \|\hat S_\eps u - u\|_{\CC^\beta}^{\alpha / \beta} \|\hat S_\eps u - u\|_\infty^{(\beta - \alpha)/\beta}
\le C \eps^{{\beta - \alpha\over 4}} \|u\|_{\CC^\beta}\;.
\end{equ}
The claim now follows from \eref{e:reprSeps}.
\end{proof}

\endappendix

\bibliographystyle{Martin}
\bibliography{./refs}

\end{document}